\documentclass{article}

\usepackage{arxiv}

\usepackage[utf8]{inputenc} 
\usepackage[T1]{fontenc}    
\usepackage{hyperref}     
\usepackage{url}           
\usepackage{booktabs}     
\usepackage{amsfonts}     
\usepackage{amsthm}
\usepackage{nicefrac}    
\usepackage{microtype}  
\usepackage{lipsum}		 
\usepackage{graphicx}
\usepackage{natbib}
\usepackage{doi}
\usepackage{kpfonts}
\usepackage{multirow}
\usepackage{xcolor}

\newtheorem{theorem}{Theorem}[section]

\newtheorem{corollary}[theorem]{Corollary}
\newtheorem{definition}[theorem]{Definition}
\newtheorem{remark}[theorem]{Remark}
\newtheorem{example}[theorem]{Example}
\newtheorem{assumption}[theorem]{Assumption}
\newtheorem{problem}[theorem]{Problem}

\newcommand{\R}{\mathbb{R}}
\newcommand{\N}{\mathbb{N}}
\newcommand{\DD}{\mathcal{D}}

\newcommand{\LH}{L^2(\Omega;H_0^1(\DD;\R))}

\newcommand{\tLH}{L^{2;(M)}(\Gamma;H_0^1(\DD;\R))}

\newcommand{\wB}{B_{a_{\min}^{-1}}}
\newcommand{\wF}{F_{a_{\min}^{-1}}}
\newcommand{\twB}{\bar{B}^{(M)}_{a_{\min}^{-1}}}
\newcommand{\twF}{\bar{F}^{(M)}_{a_{\min}^{-1}}}
\newcommand{\twE}{\bar{\mathcal{E}}^{(M)}_{a_{\min}^{-1}}}
\newcommand{\tu}{\bar{u}^{(M)}}
\newcommand{\tv}{\bar{v}^{(M)}}
\newcommand{\tf}{\bar{f}^{(M)}}

\title{Deep learning methods for stochastic Galerkin approximations of elliptic random PDEs.}
\date{} 					

\author{{\hspace{1mm}Fabio Musco} \\
	Department of Mathematics\\
	University of Stuttgart\\
	\texttt{fabio.musco@ians.uni-stuttgart.de} \\
	\And
	{\hspace{1mm}Andrea Barth} \\
	Department of Mathematics\\
	University of Stuttgart\\
	\texttt{andrea.barth@ians.uni-stuttgart.de} \\
}

\renewcommand{\headeright}{}
\renewcommand{\undertitle}{}
\renewcommand{\shorttitle}{Deep learning methods for stochastic Galerkin approximations}

\begin{document}
\bibliographystyle{plain}
\setcitestyle{numbers}
\maketitle

\begin{abstract}
%

This work considers stochastic Galerkin approximations of linear elliptic partial differential equations (PDEs) with stochastic forcing terms and stochastic diffusion coefficients, that cannot be bounded uniformly away from zero and infinity.
A traditional numerical method for solving the resulting high-dimensional coupled system of PDEs is replaced by deep learning techniques. 
In order to achieve this, physics-informed neural networks (PINNs), which typically operate on the strong residual of the PDE and can therefore be applied in a wide range of settings, are considered. As a second approach, the Deep Ritz method, which is a neural network that minimizes the Ritz energy functional to find the weak solution, is employed. While the second approach only works in special cases, it overcomes the necessity of testing in variational problems while maintaining mathematical rigor and ensuring the existence of a unique solution. Furthermore, the residual is of a lower differentiation order, reducing the training cost considerably. The efficiency of the method is demonstrated on several model problems. 
\end{abstract}


\section{Introduction}
In the context of dynamical systems, which are used to describe real-world phenomena, uncertainty or measurement errors in the data are modeled mathematically by stochastic random fields or processes. This results in the generation of a stochastic or random partial differential equation (PDE). In this case, the unknown is not an element of a function space; rather, it is a measure over a function space. Consequently, the uncertainty associated with the unknown must be quantified in a goal-oriented manner. In the absence of knowledge regarding the law governing the distribution, moments are approximated.  In recent decades, a number of techniques have been developed for quantifying this uncertainty. The aforementioned methods can be broadly classified into two categories: intrusive and non-intrusive. The latter category encompasses all sampling techniques, including stochastic collocation methods and their associated variants. In this instance, a sampling strategy based on deterministic solvers is employed. Intrusive methods are based on the (generalized) polynomial chaos (PC) expansion, which employs a spectral representation in the random space to modify the dynamical system. This results in a (potentially high-dimensional and coupled) system of equations that must be solved. In this instance, the solver must be modified to accommodate the high-dimensionality of the transformed problem. On initial examination, it would appear that non-intrusive methods offer a more advantageous approach. However, sampling methods are prone to slow convergence rates, resulting in the generation of a statistical quantity. In contrast, intrusive methods provide a controlled error for the calculation of various moments of the solution. However, the high-dimensionality of the problem represents the primary challenge when employing an intrusive method. 
For an overview and comparison of both approaches, we refer to \cite{Field2015, Tuminaro2011}

In order to illustrate the basic principles of uncertainty quantification in modeling subsurface flows, we will consider a linear elliptic partial differential equation. To account for uncertainties, the resulting parameters are modeled as random fields, which in turn results in a random solution to the system. 
This system is solved by the global stochastic Galerkin method, where a spectral expansion of the random fields with respect to the stochastic inputs on a global polynomial basis is employed. This requires that the solution is smooth with respect to the stochastic inputs (see e.g.\ \cite{Babuska2007, BECK2012, Cohen2010, Nobile2008}). The spectral expansions yield a deterministic representation of the involved random fields, that are truncated at a given maximal polynomial degree (total degree) and substituted into the governing equations, meaning the stochastic Galerkin method is an intrusive technique.

A variety of techniques may be employed in order to solve the resulting coupled system, according to classical approaches.
For example, a truncated PC expansion in the stochastic space can be combined with a finite element space approximation in the spatial domain ($p \times h$-version). 
Alternatively, finite element space approximations can be utilized in the stochastic and spatial domain ($k \times h$-version).
There are various results, showing that stochastic Galerkin approximations for the classical elliptic random PDE exhibit exponential convergence with respect to the maximal polynomial degree (see e.g.\ \cite{Babuska2004, Babuska2005, Cohen2010, Schwab2011}).
Problematic for practical applications is the fast increasing dimensionality of the coupled system of equations.
Due to the curse of dimensionality, the maximal dimension of the coupled system (growing exponentially in the number of random variables and the maximal polynomial degree of the truncated PC expansion) is usually bounded by a small number or decoupling strategies are employed for numerical examples in the literature, see e.g. \cite{Ernst2011, Karniadakis2005, Mugler2011, Mugler2013, Xiu2002, Xiu2009}. 
A common decoupling procedure is based on a specific double orthogonal basis resulting in an exponentially growing number of decoupled PDEs (\cite{Babuska2004, Babuska2005, Frauenfelder2005}). However, this approach is only applicable if the driving random fields are multilinear combinations of the random variables, which is not the case for e.g.\ log-normal diffusion coefficients.

%
%

In this paper, we substitute a conventional solution technique within the spatial domain with a deep learning methodology. By doing so we are able to solve the resulting high dimensional system of fully coupled equations on a single commercially available GPU. We mitigate the curse of dimensionality in the sense that the proposed method scales by far better than comparable traditional stochastic Galerkin methods.
A plethora of methodologies exist for solving partial differential equations with neural networks (see \cite{Blechschmidt2021} for an overview). 
Significant advancements have been made in this field thanks to the introduction of physics-informed neural networks (PINNs), initially proposed by Karniadakis and colleagues in \cite{Raissi2019}, which were inspired by the work of Lagaris~\cite{Lagaris1997, Lagaris1998}. For a general introduction to PINNs and their application to PDEs we refer to \cite{Cuomo2022}.
In the context of PINNs, which typically operate on the strong residual of PDEs, we formulate the stochastic Galerkin method on the strong form of the elliptic random PDE. We follow the approach set forth in the literature, as exemplified in~\cite{Ghanem1991, Xiu2010}, and employ a Galerkin projection in the stochastic space. 
This projection guarantees that the residual, as defined by the PDE, is orthogonal to the finite-dimensional space spanned by the polynomial basis functions of the truncated PC expansion. This results in a coupled system of deterministic PDEs, in which the PC coefficient functions of the solution are solved for in a suitable Hilbert space. Additionally, there are methods for solving weak formulations of PDEs using PINNs. For further details, we refer to~\cite{Kharazmi2019}. Nevertheless, these approaches are computationally demanding, as the residual and the neural network must be integrated against a number of test functions. To aleviate this issue we propose a second distinct approach based on minimizing a functional known as the Ritz energy functional. 
We formulate the problem of finding a Galerkin projection of the weak solution in the stochastic space by minimizing a certain weighted Ritz energy functional.
As neural networks are designed to minimize a Monte Carlo integration scheme of the loss function, also called risk, during training, the Ritz energy functional is a natural choice for the risk. 
This approach was first introduced in \cite{Weinan2018}, under the name Deep Ritz method, and overcomes the necessity of testing in variational problems while maintaining mathematical rigor and ensuring the existence of a unique weak solution to the system. 

Neural networks have been employed to solve high-dimensional stochastic PDEs. We refer to \cite{Nabian2019, Zhang2020, Zhang2019} for different approaches and architectures of the neural networks.


\subsection{Contribution of the paper}
Our primary contribution is the development of the \textit{S-GalerkinNet} and the \textit{S-RitzNet}.
The former is based on the strong form stochastic Galerkin method, whereas the latter is predicated upon the weighted weak form. In both cases, the spectral coefficients of the stochastic Galerkin approximation are approximated. In the case of the \textit{S-RitzNet}, we derive a training strategy that is based on a weighted Ritz energy functional,
which has the advantage of mathematically guaranteeing the existence of a unique weak solution (see  Theorem~\ref{Thm:RitzEnergyFunction}). 
In contrast to the \textit{S-GalerkinNet}, lower regularity assumptions are necessary on the solution (and therefore also on the random field coefficient) -- formulated in Assumption~\ref{Assumption:Requirements_ExUnique_WeakSol_REPDE}. In particular, it is not necessary to assume that the random coefficient has a square integrable weak derivative which is needed for the \textit{S-GalerkinNet} (Assumption~\ref{Assumption:additional_assm_strong_residual}). This makes the Ritz approach suitable for applications such as subsurface flow models with log-Gaussian random fields. Furthermore, no testing in training is required for the \textit{S-RitzNet}, and the weighted Ritz energy functional provides a natural loss function. However, in contrast to the \textit{S-GalerkinNet}, the training loss is not quantifiable, and a validation is employed instead. To the best of our knowledge this is the first time that the Deep Ritz method is used as a basis for a stochastic Galerkin approximation. Alleviating the main two issues of the PINN appoach: First, that one needs restrictive regularity assumptions even to formulate the residual and second, that there is no mathematical guarantee that the solution has sufficient regularity (essentially $H^2$-regularity).

In Section~\ref{Chapter:REPDE}, we derive the various notions of solutions and investigate their basic properties, with a particular focus on the weighted weak formulation. In Section~\ref{Chapter:SpectralExpansion}, we introduce the PC expansion as a basis for the stochastic Galerkin method. Theorem~\ref{Thm:RitzEnergyFunction} demonstrates that the Ritz energy functional possesses a unique minimum, which serves to solve the weighted weak form. This result is fundamental to the \textit{S-RitzNet}. We derive the computational forms of the strong and weighted weak approximations before introducing the neural network architectures, namely \textit{S-GalerkinNet} and \textit{S-RitzNet,} in Section~\ref{sec:DL}. In Section~\ref{sec:num}, we present numerical experiments that demonstrate the enhanced training efficiency of the \textit{S-RitzNet} in comparison to the \textit{S-GalerkinNet}. The initial example demonstrates the applicability of the methods for problems that necessitate a high polynomial degree (considering a one-dimensional problem with stochastic forcing). The subsequent example, a two-dimensional problem with a stochastic coefficient, illustrates the efficacy of the methods for high-dimensional systems of fully coupled equations. Finally, an equation with a log-normal coefficient is solved, which requires a high polynomial degree and a high stochastic dimension. In essence, we have developed a methodology that enables the solution of a fully coupled system of equations of dimension $120$ in just over $20$ minutes (on a single NVIDIA RTX 3070 GPU), a feat that has to our knowledge not been achieved previously in the literature. The \textit{S-RitzNet} represents a pioneering attempt to make a stochastic Galerkin approach applicable to a diverse range of real-world applications characterized by low regularity.

\section{Elliptic random partial differential equation} \label{Chapter:REPDE}


Let $(\Omega, \mathcal{A}, P)$ be a complete probability space and let $\mathcal{D} \subset \mathbb{R}^d$, $d \in \mathbb{N}$, be an open, bounded and connected domain with Lipschitz boundary $\partial \mathcal{D}$ and closure $\overline{\mathcal{D}}$. We consider the elliptic random partial differential equation (RPDE) with homogeneous Dirichlet boundary conditions, stated as follows:
\begin{problem}[Strong formulation of the elliptic RPDE] \label{Problem:StrongFormulation_REPDE}
Find the solution $u: \Omega \times \overline{\mathcal{D}} \to \mathbb{R}$, such that:
\begin{align*}
- \nabla \cdot (a(\omega,x) \, \nabla u(\omega,x) ) &= f(\omega, x) 
	&& \hspace*{-2.5cm} x \in \DD, \, \omega \in \Omega,	
\\
u(\omega, x) &= 0
	&& \hspace*{-2.5cm} x \in \partial \DD, \, \omega \in \Omega,
\end{align*}
where the arising differential operators are understood with respect to the spatial variable $x \in \mathcal{D}$.
\end{problem} 
The random field $a: \Omega \times \mathcal{D} \to \mathbb{R}$, modeling the permeability, is called \textit{stochastic diffusion coefficient} and the random field $f:\Omega \times \mathcal{D} \to \mathbb{R}$, representing a forcing or source, is called \textit{stochastic forcing term}.
We equip the domain $\mathcal{D}$ with the Borel $\sigma$-algebra $\mathcal{B}(\mathcal{D})$ and the Lebesgue-measure $\lambda^d$ and assume the random fields $a$ and $f$ to be $(\mathcal{A} \otimes \mathcal{B}(\mathcal{D})) - \mathcal{B}(\mathbb{R})$-measurable. 
For simplicity we omit the $\sigma$-algebras in measurability statements when the context is clear.
We recall the definitions of the involved function spaces. Let $p \in (0, \infty]$, and let
\begin{align*}
L^p(\DD; \R) := L^p((\DD , \mathcal{B}(\DD), \lambda^d); (\R , \mathcal{B}(\R))
	= \left\lbrace \phi : \DD \to \R \text{ measurable} \, : \, \lVert \phi \rVert_{L^p(\DD; \R)} < \infty \right\rbrace, 
\end{align*}
denote the Lebesgue space with
\begin{align*}
\lVert \phi \rVert_{L^p(\DD; \R)} = \left( \int_\DD \lvert \phi(x) \rvert^p \, d\lambda^d(x) \right)^{1/p},
\end{align*} 
for $p \in (0, \infty)$, and
\begin{align*}
\lVert \phi \rVert_{L^{\infty}(\DD; \R)} = \inf \{c \geq 0 \, : \, \lvert \phi(x) \rvert \leq c \text{ for } \lambda^d \text{-a.e.\ } x \in \DD\},
\end{align*}
for $p = \infty$. Note that $\lVert \cdot \rVert_{L^p(\DD;\R)}$ defines a norm for $p \geq 1$, and a quasi-norm for $p \in (0,1)$. 
Furthermore, let $k \in \N$ and let
\begin{align*}
H^k(\DD; \R) := H^k((\DD, \mathcal{B}(\DD), \lambda^d); (\R; \mathcal{B}(\R))
= \{ \phi \in L^2(\DD; \R) \, : \, D^\alpha \phi \in L^2(\DD; \R) \text{ for every } \alpha \in \N^d \text{ with } \lvert \alpha \rvert \leq k\},
\end{align*}
denote the $k$-th order Sobolev space, where $\alpha = (\alpha_1, \dots , \alpha_d) \in \N^d$ denotes a multiindex of order $\lvert \alpha \rvert = \alpha_1 + \dots + \alpha_d$, and 
\begin{align*}
D^\alpha \phi = \frac{\partial^{\alpha_1}}{\partial x_1^{\alpha_1}} \dots \frac{\partial^{\alpha_d}}{\partial x_d^{\alpha_d}} \phi
\end{align*}
denotes the $\alpha$-th weak partial derivative of $\phi$. The defined Sobolev space, equipped with the norm
\begin{align*}
\lVert \phi \rVert_{H^k(\DD; \R) } = 
\bigg( \sum_{\lvert \alpha \rvert \leq k} \int_\DD \lvert D^\alpha \phi \rvert^2 \, d\lambda^d(x) \bigg)^{1/2},
\end{align*}
and corresponding inner product, is a Hilbert space. In the following, we denote the weak derivatives by the usual derivative notation.
We further introduce the first-order Sobolev space of functions that vanish at the boundary in the sense of traces via
\begin{align*}
H_0^1(\mathcal{D}; \R) := H_0^1((\mathcal{D}, \mathcal{B}(\mathcal{D}), \lambda^d) ; (\R, \mathcal{B}(\R))
	= \left\lbrace \phi \in H^1(\DD; \R) \, : \,  \phi\big|_{\partial \DD} = 0 \text{ in the sense of traces}\right\rbrace,
\end{align*}
equipped with the (semi-) norm
\begin{align*}
\lVert \phi \rVert_{H_0^1(\DD; \R)} = \left( \int_\DD \lVert \nabla \phi(x) \rVert_2^2 \, d\lambda^d(x) \right)^{1/2},
\end{align*}
where $\lVert \cdot \rVert_2$ denotes the Euclidean norm on $\mathbb{R}^d$. Note that this defines a norm on $H_0^1(\DD; \R)$ due to the Poincaré inequality. The space $H_0^1(\DD; \R)$, equipped with the inner product
\begin{align*}
\langle \phi, \psi \rangle_{H_0^1(\DD; \R)} = \int_\DD \nabla \phi(x) \cdot \nabla \psi(x) \, d\lambda^d(x),
\end{align*}
is a separable Hilbert space. 
We denote by $H^{-1}(\mathcal{D}; \R)$ the dual space of $H_0^1(\mathcal{D}; \R)$, and work on the Gelfand triplet 
\begin{align*}
H_0^1(\mathcal{D}; \R) \subset L^2(\mathcal{D}; \R) \cong \big(L^2(\mathcal{D} ; \R)\big)^* \subset H^{-1}(\mathcal{D}; \R),
\end{align*}
by identifying $L^2(\DD;\R)$ with its dual space $(L^2(\mathcal{D} ; \R))^*$. For $\phi \in H_0^1(\DD; \R), \psi \in H^{-1}(\DD; \R)$, consider the dual pairing
\begin{align*}
{}_{H^{-1}(\mathcal{D}; \R)} \langle \psi, \phi \rangle_{H_0^1(\mathcal{D};\R)} = \int_\DD \phi(x) \psi(x) \, d\lambda^d(x).
\end{align*}
Furthermore, we extend the notion of real-valued random variables to those taking values in function spaces.
Let $(E, \lVert \cdot \rVert_E)$ be a Banach space and let $\mathcal{B}(E)$ denote the Borel $\sigma$-algebra on $E$, that is the smallest $\sigma$-algebra generated by the $\lVert \cdot \rVert_E$-open sets. 
We call a mapping $X: \Omega \to E$ \textit{strongly measurable random field}, if it is $\mathcal{A}-\mathcal{B}(E)$-measurable and has a separable range, i.e. $X(\Omega) \subset E$ is $\lVert \cdot \rVert_E$-separable. 
Let $p \in [1, \infty]$, and let
\begin{align*}
L^p(\Omega; E) := L^p((\Omega, \mathcal{A}, P); (E, \mathcal{B}(E)))
	= \left\lbrace \phi: \Omega \to E \text{ strongly measurable} \, : \, \lVert \phi \rVert_{L^p(\Omega; E)} < \infty \right\rbrace,
\end{align*}
denote the Lebesgue-Bochner space with the norm
\begin{align*}
\lVert \phi \rVert_{L^p(\Omega; E)} 
	&=
\left( \int_\Omega \lVert \phi(\omega) \rVert_E^p \, dP(\omega) \right)^{1/p}, \text{ for } p \in [1, \infty), 
	\\
\lVert \phi \rVert_{L^{\infty}(\Omega; E)} 
	&=
\inf \{c \geq 0 \, : \, \lVert \phi(\omega) \rVert_E \leq c \text{ for } P-\text{a.e. } \omega \in \Omega \}. 
\end{align*}
If $E$ is a separable Hilbert space, then the Lebesgue-Bochner space $L^2(\Omega; E)$ is a separable Hilbert space that is isomorphic to the tensor product Hilbert space $L^2(\Omega; E) \cong L^2(\Omega; \R) \otimes E$ (see \citep[Remark 2.19]{Schwab2011}).
We refer to \citep[App. E]{Cohn2013} for the derivation and more details on the Lebesgue--Bochner spaces and integrals. 

For many applications, as well as benchmark settings, one is interested in diffusion coefficients that cannot be uniformly bounded away from zero and infinity by deterministic constants. This is for example the case for log-normal random fields.
In this case, we cannot directly apply the classical theory of Lax-Milgram (see e.g.\ \citep[Ch. 6.2]{Evans2022}) to proof the existence of a unique weak solution to Problem \ref{Problem:StrongFormulation_REPDE}. 
Instead we follow \citep{Mugler2013}, and earlier works \citep{Charrier2013, Gittelson2010, Mugler2011}, and only assume that the diffusion coefficient can be bounded by real-valued random variables.
\begin{assumption} \label{Assumption:Requirements_ExUnique_WeakSol_REPDE}
Let the stochastic diffusion coefficient $a: \Omega \times \DD \to \R$ be in $L^2(\Omega; L^{\infty}(\DD; \R))$ and let $a_{\min}, \, a_{\max}: \Omega \to \R$ be $\sigma(a)-\mathcal{B}(\R)$-measurable random variables on $(\Omega, \mathcal{A}, P)$, with $a_{\min}^{-1} \in L^p(\Omega; \R)$ for every $p \in (0,\infty)$, such that
\begin{align*}
0 < a_{\min}(\omega) \leq a(\omega, x) \leq a_{\max}(\omega) < \infty
\end{align*}
$P$-almost surely and for $\lambda^d$-almost every $x \in \DD$. Furthermore let the stochastic forcing term $f: \Omega \times \DD \to \R$ be an element of $L^2(\Omega; H^{-1}(\DD; \R))$.
\end{assumption}
Under Assumption \ref{Assumption:Requirements_ExUnique_WeakSol_REPDE}, it immediately follows that every realization of the stochastic diffusion coefficient is uniformly bounded away from zero and infinity. 
Hence, under standard arguments, the pathwise bilinear form $\widehat{B}: H_0^1(\DD; \R) \times H_0^1(\DD; \R) \to \R$, defined by
\begin{align*}
\widehat{B}(u,v; \omega) := \int_\DD a(\omega, x) \nabla u(x) \cdot \nabla v(x) \, d\lambda^d(x),
\end{align*}
is continuous and coercive. 
This in turn guarantees, by the Lax--Milgram Lemma (see e.g.\ \citep[Ch. 6.2, Thm. 1]{Evans2022}), the existence and uniqueness of a solution to the pathwise weak formulation of Problem \ref{Problem:StrongFormulation_REPDE}, stated as follows, for any $\omega \in \Omega$.
\begin{problem}[Pathwise weak formulation of the elliptic RPDE] \label{Problem:PathwiseWeakFormulation_REPDE}
For $\omega \in \Omega$, and a realization of the forcing term $f(\omega) \in H^{-1}(\DD; \R)$ find $u(\omega) \in H_0^1(\DD; \R)$, such that
\begin{align*}
\widehat{B}(u(\omega), v; \omega) = \widehat{F}(v;\omega) \text{ for every } v \in H_0^1(\DD; \R),
\end{align*}
where $\widehat{F}(v;\omega) := \int_\DD f(\omega, x) v(x) \, d\lambda^d(x)$.
\end{problem}

If the stochastic diffusion coefficient satisfies Assumption \ref{Assumption:Requirements_ExUnique_WeakSol_REPDE} for deterministic bounds $a_{\min}, a_{\max} \in \R$, we can define the classical stochastic weak formulation of Problem \ref{Problem:StrongFormulation_REPDE} and ensure the existence and uniqueness of a weak solution. For that, we define the classical bilinear form $B: L^2(\Omega; H_0^1(\DD; \R)) \times L^2(\Omega; H_0^1(\DD; \R)) \to \R$, and the classical linear form $F: L^2(\Omega; H_0^1(\DD; \R)) \to \R$, via
\begin{align*}
B(u,v) &:= \int_\Omega \int_\DD a(\omega, x) \nabla u(\omega, x) \cdot \nabla v(\omega, x) \, d\lambda^d(x) \, dP(\omega), 
\\
F(v) &:= \int_\Omega \int_\DD f(\omega, x) v(\omega, x) \, d\lambda^d(x) \, dP(\omega).
\end{align*}
The corresponding classical stochastic weak formulation reads:
\begin{problem}[Classical stochastic weak formulation of the elliptic RPDE] \label{Problem:ClassicalWeakStochasticFormulation}
For a given stochastic forcing term $f \in L^2(\Omega; H^{-1}(\DD; \R))$ find $u \in L^2(\Omega; H_0^1(\DD; \R))$, such that
\begin{align*}
B(u,v) = F(v) \text{ for every } v \in L^2(\Omega; H_0^1(\DD; \R)).
\end{align*}
\end{problem}
We refer to \cite{Babuska2004}, \cite{Babuska2005} for more details on the elliptic random partial differential equation and the classical setup. 

For non-degenerate stochastic bounds $a_{\min}, a_{\max}: \Omega \to \R$ in Assumption \ref{Assumption:Requirements_ExUnique_WeakSol_REPDE}, the classical stochastic weak formulation, Problem \ref{Problem:ClassicalWeakStochasticFormulation}, is in general not well-posed. 
Furthermore, there are examples where stochastic Galerkin approximations, derived by this formulation whether it is well-posed or not, don't converge to the true solution in the natural norm (see e.g.\ \citep[Ch. 4]{Mugler2013}). 
In order to obtain a well-posed formulation that produces convergent stochastic Galerkin approximations, we follow \citep{Mugler2013} and work on weighted spaces. Let $\rho: \Omega \times \DD \to \R$ be a measurable random field that is $P$-almost surely and $\lambda^d$-almost everywhere positive. We define the $\rho$-weighted Lebesgue-Bochner space via
\begin{align*}
L_{\rho}^2(\Omega; H_0^1(\DD; \R)) := \left\lbrace\phi: \Omega \to H_0^1(\DD;\R) \text { strongly measurable} \, : \, \lVert \phi \rVert_{L_{\rho}^2(\Omega; H_0^1(\DD; \R))} < \infty \right\rbrace,
\end{align*} 
where the $\rho$-weighted norm is defined by
\begin{align*}
\lVert \phi \rVert_{L_{\rho}^2(\Omega; H_0^1(\DD; \R))} := 
	\left( \int_\Omega \int_\DD \lVert \nabla \phi(\omega, x) \rVert_2^2 \, \rho(\omega, x) \, d\lambda^d(x) \, dP(\omega) \right)^{1/2}. 
\end{align*}
The space $L_{\rho}^2(\Omega; H_0^1(\DD; \R))$, equipped with the inner product
\begin{align*}
\langle \phi , \psi \rangle_{L_{\rho}^2(\Omega; H_0^1(\DD;\R))}
:= \int_\Omega \int_\DD \nabla \phi(\omega, x) \cdot \nabla \psi(\omega, x) \, \rho(\omega, x) \, d\lambda^d(x) \, dP(\omega),
\end{align*}
is a separable Hilbert space. If $\rho: \Omega \to \R$ is a $P$-almost surely positive, spatial independent random variable, we define the $\rho$-weighted Lebesgue-Bochner space $L^2_\rho(\Omega; H)$, for arbitrary separable Hilbert spaces $H$, by
\begin{align*}
L^2_\rho(\Omega;H) := \left\lbrace \phi: \Omega \to H \text{ strongly measurable} \, : \, \lVert \phi \rVert_{L_{\rho}^2(\Omega; H)} < \infty \right\rbrace,
\end{align*}
where the norm is given by
\begin{align*}
\lVert \phi \rVert_{L_{\rho}^2(\Omega; H)}
= \left( \int_\Omega \lVert \phi(\omega) \rVert_H^2 \rho(\omega) \, dP(\omega) \right)^{1/2}.
\end{align*}
The space $L_{\rho}^2(\Omega; H)$, equipped with the inner product
\begin{align*}
\langle \phi , \psi \rangle_{L_{\rho}^2(\Omega; H)}
:= \int_\Omega \langle \phi(\omega), \psi (\omega) \rangle_H \, \rho(\omega) \, dP(\omega),
\end{align*}
is a separable Hilbert space, that is isomorphic to the tensor product Hilbert space
\begin{align*}
L_{\rho}^2(\Omega ; H) \cong L_{\rho}^2(\Omega; \R) \otimes H.
\end{align*}
Analogously, we define the $a_{\min}^{-1}$-weighted operators on the $a/a_{\min}$-weighted Lebesgue-Bochner space via
\begin{align*}
& B_{a_{\min}^{-1}}:
L_{a/a_{\min}}^2(\Omega; H_0^1(\DD; \R)) \times L_{a/a_{\min}}^2(\Omega; H_0^1(\DD; \R)) \to \R, 
	\\
	& \hspace*{0.9cm}
(u,v) \mapsto \int_\Omega \int_\DD \frac{a(\omega, x)}{a_{\min}(\omega)} \nabla u(\omega, x) \cdot \nabla v(\omega, x) \, d\lambda^d(x) \, dP(\omega), 
	\\
& F_{a_{\min}^{-1}} : L_{a/a_{\min}}^2(\Omega; H_0^1(\DD; \R)) \to \R
	\\
	& \hspace*{0.9cm}
v \mapsto \int_\Omega \int_\DD \frac{1}{a_{\min}(\omega)} f(\omega, x) v(\omega, x) \, d\lambda^d(x) \, dP(\omega).
\end{align*}
Note that the random variable $a_{\min}: \Omega \to \R$ stems from Assumption \ref{Assumption:Requirements_ExUnique_WeakSol_REPDE} and the weighted bilinear form $B_{a_{\min}^{-1}}$ is continuous and coercive on $L_{a/a_{\min}}^2(\Omega; H_0^1(\DD; \R)) \times L_{a/a_{\min}}^2(\Omega; H_0^1(\DD; \R))$. 
Furthermore, if $f \in L_{a_{\min}^{-2}}^2(\Omega; H^{-1}(\DD; \R))$, then the weighted linear form $F_{a_{\min}^{-1}}$ is continuous on $L_{a/a_{\min}}^2(\Omega; H_0^1(\DD; \R))$. Hence, by the Lax-Milgram Lemma, we get the existence and uniqueness of a solution to the following problem:
\begin{problem}[Weighted stochastic weak formulation of the elliptic RPDE] \label{Problem:WeightedWeakStochasticFormulation}
For a given stochastic forcing term $f \in L_{a_{\min}^{-2}}^2(\Omega; H^{-1}(\DD; \R))$ find $u \in L_{a/a_{\min}}^2(\Omega; H_0^1(\DD; \R))$, such that
\begin{align*}
B_{a_{\min}^{-1}}(u,v) = F_{a_{\min}^{-1}}(v) \text{ for every } v \in L_{a/a_{\min}}^2(\Omega; H_0^1(\DD; \R)).
\end{align*}
\end{problem}

\begin{remark}
For $\omega \in \Omega$, the realization $u(\omega, \cdot)$ of the weak solution $u \in L_{a/a_{\min}}^2(\Omega; H_0^1(\DD; \R))$ of the weighted stochastic formulation, Problem \ref{Problem:WeightedWeakStochasticFormulation}, is almost surely equal to the pathwise solution $\widehat{u}(\omega) \in H_0^1(\DD; \R)$ of the pathwise weak formulation, Problem \ref{Problem:PathwiseWeakFormulation_REPDE}, i.e.
\begin{align*}
u(\omega, \cdot) \equiv \widehat{u}(\omega) \text{ in } \DD.
\end{align*}
When considering deterministic bounds $a_{\min}, \, a_{\max} \in \R$ in Assumption \ref{Assumption:Requirements_ExUnique_WeakSol_REPDE}, the weak solution $u \in L_{a/a_{\min}}^2(\Omega; H_0^1(\DD; \R))$ coincides with the weak solution $\tilde{u} \in \LH$ of the classical stochastic formulation, Problem \ref{Problem:ClassicalWeakStochasticFormulation}.
\end{remark}

Based on these notions of solution we introduce different stochastic Galerkin approximations in the next section.
\section{Spectral expansion methods} \label{Chapter:SpectralExpansion}
\subsection{Generalized Polynomial Chaos Expansion} \label{Section:gPC}
A polynomial chaos (PC) expansion of a function, driven by underlying stochasticity, represents its spectral expansion in the stochastic space with respect to an orthonormal polynomial basis. Any random field in $L^2(\Omega;H)$, for a separable Hilbert space $H$, can be expanded in such a fashion.
For that, any orthogonal polynomial basis of $L^2(\Omega;\R)$ is suitable. 
Such a spectral expansion with Hermite polynomials, evaluated in a sequence of independent Gaussian random variables, was first introduced by Wiener (see \citep{Wiener1938}) and is called Hermite expansion.
In some cases, especially when the random field being expanded is not Gaussian, the convergence of the Hermite expansion can be quite slow. 
In \citep{XiuWienerAskey}, Xiu and Karniadakis proposed that using random variables following distributions that are closer to that of the random field can oftentimes result in a much better rate of convergence.
In this case, a polynomial basis is chosen that is orthogonal with respect to the law of these random variables.
For details, we refer to \citep{Karniadakis2005, Xiu2002, XiuWienerAskey}. This polynomial basis is called generalized polynomial chaos (gPC) basis.
A theoretical treatment of the convergence, and existence of such an improved basis, as well as examples where this method is not applicable can be found in \citep{Ernst2011}. 
 
In applications, one is usually interested in functions that are driven by a finite number of random perturbations, where possible truncations arise, for example, from truncated Karhunen-Loève-type expansions (see e.g.\ \citep[Ch. 2.3]{Ghanem1991}). 
In that context, we assume the arising random fields to be driven by finitely many random variables. 
Let $(\Omega, \mathcal{A}, P)$ be a complete probability space, and let $(\Gamma_n, \Sigma_n)$, for every $n = 1, \dots , N$ with $N \in \mathbb{N}$, be a measurable space for Borel subsets $\Gamma_n$ of $\mathbb{R}$, equipped with a corresponding $\sigma$-algebra $\Sigma_n$. Furthermore, let $\Gamma := \varprod_{n = 1}^N \Gamma_n$ denote the product space, endowed with the corresponding product $\sigma$-algebra $\Sigma := \bigotimes_{n= 1}^N \Sigma_n$.
\begin{definition} \label{Def:UnderlyingRVs_and_Distribution}
Let $(Y_n)_{n = 1, \dots , N}$ denote a sequence of independent, $\mathcal{A}-\Sigma_n$-measurable, continuous random variables $Y_n:\Omega \to \Gamma_n$, with law
\begin{align*}
\mu_n := P \circ Y_n^{-1} : \Sigma_n \to [0,1].
\end{align*}
For this sequence, we define the $\mathcal{A}-\Sigma$-measurable map $Y: \Omega \to \Gamma$ via
$Y(\omega) = (Y_1(\omega), \dots , Y_N(\omega))$. 
\end{definition}
Due to the independence of the random variables $(Y_n)_{n = 1, \dots, N}$, the law of $Y$ is given by the product measure $\mu = \bigotimes_{n=1}^N \mu_n$.
We refer to \cite[Ch. 5, Ch. 10.6]{Cohn2013} for more details on product spaces and measures.
To obtain a PC expansion that is evaluated in the sequence $(Y_n)_{n = 1, \dots , N}$, we construct a polynomial basis of $L^2(\Gamma; \R):= L^2((\Gamma, \Sigma, \mu); (\R, \mathcal{B}(\R))$, that is orthonormal with respect to $\mu$. In order to do so, we require the random variables $(Y_n)_{n = 1, \dots , N}$ to meet the following assumptions:

\begin{assumption} \label{Assumption:RVsReq_gPC_MomentsAndDeterminateMeasure} \,
\vspace*{-0.1cm}
\begin{itemize} 
\item[(i)] For every $n = 1, \dots, N$, the random variable $Y_n$ has finite moments, i.e.\ for every $k \in \mathbb{N}$ the $k$-th moment, given by $\mathbb{E}(Y_n^k) = \int_\Omega Y_n^k(\omega) \, dP(\omega) = \int_{\Gamma_n} y^k \, d\mu_n(y)$, is finite.
\item[(ii)] For every $n = 1, \dots , N$, the distribution $\mu_n$ is determinate, i.e.\ the measure $\mu_n$ is uniquely determined by its moments $\big(\mathbb{E}(Y_n^k)\big)_{k \in \mathbb{N}}$, meaning no other probability distribution possesses an identical sequence of moments. 
\end{itemize}
\end{assumption}
Note that in order to construct an orthonormal system of polynomials in $L^2(\Gamma; \R)$, the first assumption suffices (see \cite[Asm. 3.1]{Ernst2011} and thereafter).
For the space $\Gamma$ of finite products, there are conditions for which the measure $\mu$ is determinate, this is for example the case for uniform and normal distributions (see \cite[Thm.\ 3.4, Thm.\ 3.7]{Ernst2011}). 
The log-normal distribution however is not determinate, so the corresponding orthonormal system of polynomials does not constitute a basis of $L^2(\Gamma; \R)$. 
Therefore, there is no gPC expansion for log-normal random fields, and they are usually expanded in Hermite polynomials, evaluated in a sequence of independent standard normal distributed random variables. In applications, the expansion can only be truncated at a high polynomial degree due to the slow convergence of the Hermite expansion for log-normal random fields. This results in a high-dimensional system, which entails high computational costs.

The basis polynomials of $L^2(\Gamma_n ; \R)$ can be constructed by applying the Gram-Schmidt orthonormalization (\cite[Eq. (1.1.11)]{Gautschi2004}) to the monomials with respect to corresponding measures. This results in a convenient three term recursion, and vice versa, meaning that polynomials constructed by a certain three term recursion result in an orthonormal polynomial basis (see e.g.\ \citep[Ch.\ 1.3]{Gautschi2004}). 
Basis polynomials that are orthonormal with respect to certain probability laws are tabulated for many common distributions, see for example \cite[Tbl. 1.1, Tbl. 1.2]{Gautschi2004}, 
and \cite[Tbl. 4.1]{XiuWienerAskey}. Here we mention the Legendre polynomials for the uniform distribution on the interval $[-1,1]$, $\mathcal{U}([-1,1])$, and the probabilist's Hermite polynomials for the standard normal distribution, $\mathcal{N}(0,1)$. 
For more details regarding the existence and construction of orthonormal polynomial bases, we refer to
\citep[Thm. 1.6, Thm. 1.27, Thm. 1.29]{Gautschi2004},
\cite[Ch.\ 2.2]{Schwab2011}.

Let $(p_{n,i})_{i \in \mathbb{N}_0}$, for $n = 1, \dots , N$, be an orthonormal polynomial basis of $L^2(\Gamma_n; \R)$, with respect to the law $\mu_n$ of the random variable $Y_n$ in which the basis is evaluated, i.e.\
\begin{align*}
\langle p_{n,i} , p_{n,j} \rangle_{L^2(\Gamma_n ; \R)} = \int_{\Gamma_n} p_{n,i}(y) \, p_{n,j}(y) \, d \mu_n(y) = \delta_{ij},
\end{align*}
where $\delta_{ij}$ denotes the Kronecker-Delta. 
Therefore, every function $\phi \in L^2(\Gamma_n; \R)$ admits the spectral expansion
\begin{align*}
\phi(y) = \sum_{i \in \mathbb{N}_0} \langle \phi , p_{n,i} \rangle_{L^2(\Gamma_n;\R)} \, p_{n,i}(y), \quad y \in \Gamma_n,
\end{align*}
with convergence in $L^2(\Gamma_n;\R)$.
Every basis function $p_{n,i}: \Gamma_n \to \mathbb{R}$, constructed by the above-mentioned and referenced techniques, is a polynomial of degree $i \in \mathbb{N}_0$, and satisfies $p_{n,0} \equiv 1$ for every $n = 1, \dots , N$. 
Define the set of index sequences $\mathcal{I}_N := \{ \nu = (\nu_1, \dots , \nu_N)  \in \mathbb{N}_0^N \, : \, \nu_i \in \mathbb{N}_0 \}$, and for every index sequence $\nu \in \mathcal{I}_N$ a corresponding tensor product function
\begin{align*}
p_\nu := \bigotimes_{n=1}^N p_{n, \nu_n}: \Gamma \to \mathbb{R}, \quad y=(y_1, \dots , y_N) \mapsto \prod_{n = 1}^N p_{n, \nu_n}(y_n).
\end{align*}
For every $\nu \in \mathcal{I}_N$, the tensor product function $p_\nu$ is a multivariate polynomial of degree 
$\lvert \nu \rvert := \nu_1 + \dots + \nu_N$,
 and constitutes a basis of the product space $L^2(\Gamma; \R)$, according to the following theorem (see e.g.\ \citep[Thm.\ 3.6]{Ernst2011}):

\begin{theorem} \label{Theorem:gPC_basis}
Let $Y = (Y_1, \dots , Y_N)$, for $N \in \mathbb{N}$, be the random variable defined in Definition \ref{Def:UnderlyingRVs_and_Distribution}, such that the sequence $(Y_n)_{n = 1, \dots , N}$ satisfies Assumption \ref{Assumption:RVsReq_gPC_MomentsAndDeterminateMeasure}, and let $(p_{n,i})_{i \in \mathbb{N}_0}$, for $n = 1, \dots , N$, be the corresponding orthonormal polynomial basis of $L^2(\Gamma_n; \R)$. Then, the tensor product polynomials $(p_\nu)_{\nu \in \mathcal{I}_N} $ form an orthonormal polynomial basis of $L^2(\Gamma; \R)$, which is called the \textit{polynomial chaos basis}.
\end{theorem}
Note that, if the expanded random fields are driven directly by $Y: \Omega \to \Gamma$, used in Theorem \ref{Theorem:gPC_basis}, we call the resulting basis generalized polynomial chaos basis and benefit from the faster convergence, as discussed in the introduction to this section.
As a result of Theorem \ref{Theorem:gPC_basis}, every function $g \in L^2(\Gamma; \R)$ can be represented by
\begin{align*}
g(y) = \sum_{\nu \in \mathcal{I}_N} g_\nu \, p_\nu(y), 
\end{align*}
where $(p_\nu)_{\nu \in \mathcal{I}_N}$ denotes the PC basis of $L^2(\Gamma; \R)$, and the spectral coefficients $g_\nu \in \R$ are given by
\begin{align*}
g_\nu = \langle g , p_\nu \rangle_{L^2(\Gamma; \R)} = \int_\Gamma g(y) \, p_\nu(y) \, d\mu(y). 
\end{align*}
The convergence of the abstract Fourier series is in $L^2(\Gamma; \R)$. As a direct consequence of the construction of $Y: \Omega \to \Gamma$, the PC basis $(p_\nu)_{\nu \in \mathcal{I}_N}$ of $L^2(\Gamma; \R)$, constitutes a PC basis $(p_\nu(Y))_{\nu \in \mathcal{I}_N}$ of $L^2(\Omega ; \R)$. 
For a more accessible notation, we indicate the PC basis polynomials by natural numbers $k \in \mathbb{N}_0$, instead of index sequences $\nu \in \mathcal{I}_N$. 
For this purpose, the PC basis $(p_k)_{k \in \mathbb{N}_0}$, indexed by natural numbers, which arises from the tensor product basis $(p_\nu)_{\nu \in \mathcal{I}_N}$, is determined by the so-called graded lexicographic ordering. 
The index sequences $\nu \in \mathcal{I}_N$ are ordered and then each assigned to a natural index by counting them. 
At first, the resulting polynomial degree of two index sequences are compared against each other by the grading function $\lvert \, \cdot \, \rvert: \mathcal{I}_N \to \mathbb{N}_0, \, \nu \mapsto \nu_1 + \dots + \nu_N$. 
Two index sequences that result in the same polynomial degree are ordered lexicographically. 
Formally the ordering is given by the following definition:
\begin{definition} \label{Def:graded_lexic_ordering}
An index sequence $\nu^{(1)} = \big(\nu_1^{(1)} , \dots , \nu_N^{(1)} \big) \in \mathcal{I}_N$ is said to be smaller, in the sense of the graded lexicographic ordering, than $\nu^{(2)} = \big(\nu_1^{(2)} , \dots , \nu_N^{(2)} \big)  \in \mathcal{I}_N$, if either $\lvert \nu^{(1)} \rvert < \lvert \nu^{(2)} \rvert$, or $\lvert \nu^{(1)} \rvert = \lvert \nu^{(2)} \rvert$ and $ \nu^{(1)}_{i^*} < \nu^{(2)}_{i^*}$, where $i^* := \min \{i \in \{1, \dots , N\} \, : \, \nu^{(1)}_i \neq \nu^{(2)}_i \} $. 
The basis polynomials $(p_\nu)_{\nu \in \mathcal{I}_N}$ are ordered with respect to this relation and indexed accordingly.

\end{definition}
\begin{example}
The normed probabilist's Hermite polynomials $(\operatorname{He}_k)_{k \in \mathbb{N}_0}$, given by
\begin{align*}
\operatorname{He}_k(y) = \frac{(-1)^k}{\sqrt{k!}} e^{y^2/2} \frac{d^k}{dy^k} e^{-y^2/2},
\end{align*}
form an orthonormal basis of $L^2(\R ; \R)$, with respect to the standard normal distribution $\mu = \mathcal{N}(0,1)$. 
We assume two independent underlying random variables $Y_1, Y_2 \sim \mathcal{N}(0,1)$ (i.e. $N = 2$), and truncate the polynomial chaos expansion at a maximal polynomial degree of $P = 2$. 
Therefore the chaos basis is formed on the index set $\mathcal{I}_{N,P} := \{ \nu \in \mathcal{I}_N \, : \, \lvert \nu \rvert \leq P \}$. This results in a truncated PC expansion with $\frac{(N+P)!}{N!P!} = 6$ basis polynomials, given in their graded lexicographic ordering as follows:
\begin{center}
\begin{tabular}{l || c | c | c | c | c | c}
Index sequence $\nu \in \mathcal{I}_{2,2}$ & $(0,0)$ & $(0,1)$ & $(1,0)$ & $(0,2)$ & $(1,1)$ & $(2,0)$  \\
Associated index $k \in \mathbb{N}_0$ & 0 & 1 & 2 & 3 & 4 & 5 \\
\multirow{2}{*}{Basis Polynomial $p_k(y_1,y_2)$} &
 $\operatorname{He}_0(y_1)$  &  $\operatorname{He}_0(y_1)$  &  $\operatorname{He}_1(y_1)$  &  $\operatorname{He}_0(y_1)$  &  $\operatorname{He}_1(y_1)$ &  $\operatorname{He}_2(y_1)$  \\ 
& $\hspace*{0.2cm}  \cdot \operatorname{He}_0(y_2)$ & $\hspace*{0.2cm}  \cdot \operatorname{He}_1(y_2)$ & $\hspace*{0.2cm}  \cdot \operatorname{He}_0(y_2)$ & $\hspace*{0.2cm}  \cdot \operatorname{He}_2(y_2)$ & $\hspace*{0.2cm}  \cdot \operatorname{He}_1(y_2)$ & $\hspace*{0.2cm} \cdot \operatorname{He}_0(y_2)$ 

\end{tabular}
\end{center}

\end{example}

\begin{remark} \label{Remark:gPCHilbertValuedRF}
We can extend the PC expansion of $\R$-valued random variables to certain random fields. Let $X \in L^2(\Omega;H)$ be a strongly measurable random field, let $(H, \langle \cdot , \cdot \rangle_H)$ be a separable Hilbert space, and let $(p_k(Y))_{k \in \N_0}$ be the polynomial chaos basis of $L^2(\Omega; \R)$, then the random field $X: \Omega \to H$ admits the PC expansion
\begin{align*}
X(\omega) = \sum_{k \in \N_0} x_k \, p_k(Y(\omega)),
\end{align*}
with convergence in $L^2(\Omega;H)$, and spectral coefficients $x_k \in H$, for $k \in \N_0$. This is evident from the isomorphic relation $L^2(\Omega; H) \cong L^2(\Omega; \R) \otimes H$, for a separable Hilbert space $H$ (see also \citep[Rem.\ 3.12]{Ernst2011}). 
\end{remark}

Due to the construction of the polynomial basis in the image space of the underlying random variables, we can represent the random fields by deterministic functions.

\begin{definition}  \label{Def:DetRepRandomFields}
Let $(H(\DD; \R), \langle \cdot , \cdot \rangle_{H(\DD; \R)})$ be a separable Hilbert space over the domain $\DD \subset \R^d$, consisting of functions $g: \DD \to \R$ and let $g: \Omega \times \DD \to \R, \, g \in L^2(\Omega;H(\DD; \R))$, be a random field with PC expansion
\begin{align*}
g(\omega, x) = \sum_{k \in \mathbb{N}_0} g_k(x) \, p_k(Y(\omega)),
\end{align*}
where $g_k \in H(\DD;\R)$ for all $k \in \mathbb{N}_0$. The function $\bar{g}: \Gamma \times \DD \to \R, \, \bar{g} \in L^2(\Gamma; H(\DD; \R))$, defined by
\begin{align*}
\bar{g}(y,x) = \sum_{k \in \mathbb{N}_0} g_k(x) \, p_k(y),
\end{align*}
is called deterministic representation of $g$.
\end{definition}
Note that we get for a random field $g$ and its deterministic representation $\bar{g}$ the pathwise relation
\begin{align*}
\bar{g}(Y(\omega), x) = g(\omega, x) \text{ for } P\text{-almost every } \omega \in \Omega.
\end{align*}

\subsection{Stochastic Galerkin Method}
In the first step, we expand the stochastic diffusion coefficient and forcing term into their respective polynomial chaos expansions.
We mainly focus on the strong formulation of the elliptic random PDE, Problem \ref{Problem:StrongFormulation_REPDE}, and the weighted stochastic weak formulation of the elliptic random PDE (resp.\ the classical stochastic weak formulation, Problem \ref{Problem:ClassicalWeakStochasticFormulation}, if it is well-defined). The polynomial chaos expansions used for this are truncated at a maximal polynomial degree $P \in \N_0$. 

Consider the probability space $(\Omega, \mathcal{A}, P)$ and let $Y = (Y_1, \dots, Y_N): \Omega \to \Gamma$, $\Gamma \subset \R^N$, be the multivariate random variable, meeting Assumption \ref{Assumption:RVsReq_gPC_MomentsAndDeterminateMeasure}, in which the polynomial basis is evaluated. Furthermore, let $(p_k(Y))_{k \in \N_0}$ denote the polynomial chaos basis of $L^2(\Omega;\R)$ and accordingly let $(p_k)_{k \in \N_0}$ denote the polynomial chaos basis of $L^2(\Gamma; \R)$. 
When working with the strong form of the elliptic random PDE, we have to place additional assumptions on the random fields in order for the strong residual to be meaningful.
\begin{assumption} \label{Assumption:additional_assm_strong_residual}
In addition to Assumption \ref{Assumption:Requirements_ExUnique_WeakSol_REPDE}, we also assume
\begin{align*}
a \in L^2(\Omega; H^1(\DD; \R)), \quad f \in L^2(\Omega; L^2(\DD; \R)).
\end{align*}
\end{assumption}
For the stochastic Galerkin approximation of the weak form of the elliptic  random PDE, Assumption \ref{Assumption:Requirements_ExUnique_WeakSol_REPDE} suffices for the well-posedness, and the diffusion coefficient is expanded in the space $L^2(\Omega; L^2(\DD; \R))$.
The deterministic representation of the stochastic diffusion coefficient and the stochastic forcing term read
\begin{align*}
\bar{a}(y, x) = \sum_{k \in \N_0} a_k(x) \, p_k(y), \quad
\bar{f}(y, x) = \sum_{k \in \N_0} f_k(x) \, p_k(y), \quad y \in \Gamma, \, x \in \DD,
\end{align*}
where the spectral coefficients satisfy $a_k \in H^1(\DD; \R)$ or $a_k \in L^2(\DD; \R)$ respectively, and $f_k \in L^2(\DD;\R)$ or $f_k \in H^{-1}(\DD; \R)$ for every $k \in \N_0$. 
With these assumptions, and by the Doob-Dynkin Lemma (see \citep[Lem. 1.14]{Kallenberg2021}), the solution $u:\Omega \times \overline{\DD} \to \R$ of the elliptic random PDE can also be expressed by $Y: \Omega \to \Gamma$. We choose $L^2(\Omega; H_0^2(\DD; \R))$ as solution space for the strong form and $L^2(\Omega; H_0^1(\DD; \R))$ as solution space for the weak form, and obtain the deterministic representation $\bar{u}: \Gamma \times \overline{\DD} \to \R$ of the unknown solution with respect to  $(p_k)_{k \in \N_0}$, as
\begin{align*}
\bar{u}(y, x) = \sum_{k \in \N_0} u_k(x) \, p_k(y), \quad y \in \Gamma, \, x \in \overline{\DD},
\end{align*}
where $u_k \in H_0^2(\DD; \R)$ or $u_k \in H_0^1(\DD; \R)$ for every $k \in \N_0$.
For the stochastic Galerkin approximation, the deterministic representations are truncated at a maximal polynomial degree of $P \in \N_0$. 
This results in the solution space 
\begin{align*}
L^{2;(M)}(\Gamma; H(\DD;\R)) \cong L^{2;(M)}(\Gamma; \R) \otimes H(\DD; \R),
\end{align*}
where $H(\DD; \R) \in \{H_0^1(\DD; \R), H_0^2(\DD; \R)\}$ for the respective formulations, and the finite dimensional space $L^{2;(M)}(\Gamma;\R)$ is spanned by the first $M + 1= \frac{(N + P)!}{N! P!}$ basis polynomials $(p_k)_{k \in \{ 0, \dots , M \} }$ of $L^2(\Gamma; \R)$.
We further assume that the truncated deterministic representation of the stochastic diffusion coefficient satisfies Assumption \ref{Assumption:Requirements_ExUnique_WeakSol_REPDE}, and refer to the following remark:
\begin{remark}
There are additional assumptions, for which every truncation of the Karhunen-Loève (KL) expansion of the stochastic diffusion coefficient satisfies Assumption \ref{Assumption:Requirements_ExUnique_WeakSol_REPDE}. 
Such conditions can, for example, be found in \citep[Asm. 6.1]{Babuska2004, Babuska2005} for deterministic bounds $a_{\min}, a_{\max} \in \R$, or in \citep[Prop. 2.6]{Charrier2013} for log-normal random fields. 
There are also additional assumptions, such that every truncation of the polynomial chaos expansion of the stochastic diffusion coefficient satisfies Assumption \ref{Assumption:Requirements_ExUnique_WeakSol_REPDE}. 
However, Hermite expansions of log-normal random fields are generally excluded from these assumptions, and we have to ensure the validity for every truncation used. For example, the Hermite expansion of a standard log-normal random variable $X: \Omega \to \R_+, \, \log(X) \sim \mathcal{N}(0,1)$, reads
\begin{align*}
X(\omega) = e^{1/2} \sum_{k = 0}^{\infty} \frac{1}{k!} \operatorname{He}_k(Y(\omega)),
\end{align*}
where $\operatorname{He}_k: \R \to \R$ denotes the $k$-th probabilist's Hermite polynomial and $Y \sim \mathcal{N}(0,1)$. A truncation at a maximal polynomial degree of $P = 1$ is negative with probability greater zero, whereas a truncation at a maximal polynomial degree of $P = 2$ meets Assumption \ref{Assumption:Requirements_ExUnique_WeakSol_REPDE}. 
\end{remark}
The stochastic Galerkin method, in the strong-residual form, seeks to find an approximation $\bar{u}^{(M)}: \Gamma \times \overline{\DD} \to \R$, such that the truncated, strong residual $\mathcal{R}_{\bar{u}^{(M)}} : \Gamma \times \DD \to \R$, defined by
\begin{align*}
\mathcal{R}_{\bar{u}^{(M)}}(y,x) := \nabla \cdot \left( \bar{a}^{(M)}(y,x) \nabla \bar{u}^{(M)}(y,x) \right) + \bar{f}^{(M)}(y,x),
\end{align*}
is orthogonal to the space $L^{2;(M)}(\Gamma; \R)$ for $\lambda^d$-almost every $x \in \DD$. Here,
\begin{align*}
\bar {a}^{(M)}(y,x):= \sum_{i = 0}^M a_i(x) \, p_i(y) , \quad \bar{f}^{(M)}(y,x) := \sum_{i = 0}^M f_i(x) \, p_i(y)
\end{align*}
denote the truncated deterministic representation of the diffusion coefficient and forcing term. 
This is achieved by a Galerkin projection, which results in a system of $M + 1$, usually coupled, deterministic PDEs.
\begin{problem}[Strong form stochastic Galerkin approximation] \label{Problem:StrongFormSGA}
For the truncated deterministic representation $\bar{f}^{(M)} \in L^{2;(M)}(\Gamma; L^2(\DD; \R))$ of a given stochastic forcing term, find the spectral coefficients $u_0, \dots , u_M \in H_0^2(\DD; \R)$ of $\bar{u}^{(M)}(y,x) = \sum_{i = 0}^M u_i(x) \, p_i(y) \in L^{2;(M)}(\Gamma; H_0^2(\DD; \R))$ such that for every $ k = 0, \dots, M$
\begin{align*}
\big\langle \mathcal{R}_{\bar{u}^{(M)}}(\cdot, x) \, , \, p_k \big\rangle_{L^2(\Gamma;\R)}
	&=
0 
	\qquad
 x \in \DD, 
 	\\
\big\langle \bar{u}^{(M)}(\cdot, x) \, , \, p_k \big\rangle_{L^2(\Gamma;\R)} 
	&=
0
	\qquad
 x \in \partial \DD.
\end{align*}  
\end{problem}
Note that for a stochastic Galerkin approximation $\bar{u}^{(M)}(y,x)$, i.e.\ the solution of Problem \ref{Problem:StrongFormSGA}, the random field $u^{(M)}(\omega, x) =  \bar{u}^{(M)}(Y(\omega), x) = \sum_{i = 0}^M u_i(x) \, p_i(Y(\omega)) \in L^{2;(M)}(\Omega; H_0^2(\DD; \R))$ yields a $L^{2;(M)}(\Omega; \R)$-orthogonal residual of the strong-form random elliptic PDE, Problem \ref{Problem:StrongFormulation_REPDE}.

For the stochastic Galerkin approximation of the weak formulation of the elliptic random PDE, we define the $a_{\min}^{-1}$-weighted, truncated deterministic representation of the bilinear and linear form. For this we note that since $a_{\min}^{-1} \in L^2(\Omega; \R)$ is $\sigma(a)-\mathcal{B}(\R)$-measurable, the deterministic representation $\bar{a}_{\min}^{\, -1} \in L^2(\Gamma; \R)$, with respect to the PC basis $(p_k)_{k \in \N_0}$, exists by the Doob-Dynkin Lemma. Furthermore, by Assumption \ref{Assumption:Requirements_ExUnique_WeakSol_REPDE}, we get $\frac{\bar{a}(y,x)}{\bar{a}_{\min}(y)} \in L^2(\Gamma; L^2(\DD; \R))$. Therefore, we directly expand the function $\bar{\mathbf{a}}(y,x) := \frac{\bar{a}(y,x)}{\bar{a}_{\min}(y)}$ into the PC expansion, and get
\begin{align*}
& \twB: \tLH \times \tLH \to \R,
	\\
	& \hspace*{0.9cm}
(\tu, \tv) \mapsto 
\int_\Gamma \int_\mathcal{D} \bar{\mathbf{a}}^{(M)}(y,x) \, \nabla \tu (y,x) \cdot \nabla \tv (y,x) \, d \lambda^d(x) \, d\mu(y),
	\\
	&
\twF : \tLH \to \R,
	\\
	& \hspace*{0.9cm}
\tv \mapsto 
\int_\Gamma \int_\DD (\bar{a}_{\min}^{\, -1})^{(M)}(y) \tf (y, x) \tv (y, x)\, d\lambda^d(x) \, d\mu(y).
\end{align*}
Note that the unweighted deterministic representations of the operators $\bar{B}^{(M)}, \, \bar{F}^{(M)}$ are defined analogously, where the weight $a_{\min}^{-1}$ is replaced by $1$.
With that, we formulate the following two problems:
\begin{problem}[Classical weak form stochastic Galerkin approximation] \label{Problem:ClassicalWeakSGA}
For the truncated deterministic representation $\tf \in L^{2; (M)}(\Gamma; H^{-1}(\DD; \R))$ of a given stochastic forcing term, find the spectral coefficients $u_0, \ldots , u_M \in H_0^1(\DD; \R))$ of $\tu (y,x) = \sum_{i = 0}^M u_i(x) p_i(y) \in \tLH$ such that
\begin{align*}
\bar{B}^{(M)} (\tu , \tv ) = \bar{F}^{(M)} (\tv) \text{ for all } \tv \in \tLH.
\end{align*}
\end{problem}

\begin{problem}[Weighted weak form stochastic Galerkin approximation] \label{Problem:WeightedWeakSGA}
For the truncated deterministic representation $\tf \in L^{2; (M)}(\Gamma; H^{-1}(\DD; \R))$ of a given stochastic forcing term, find the spectral coefficients $u_0, \ldots , u_M \in H_0^1(\DD; \R))$ of $\tu (y,x) = \sum_{i = 0}^M u_i(x) p_i(y) \in \tLH$ such that
\begin{align*}
\twB (\tu , \tv ) = \twF (\tv) \text{ for all } \tv \in \tLH.
\end{align*}
\end{problem}
Note that, in contrast to the weighted stochastic weak formulation, Problem \ref{Problem:WeightedWeakStochasticFormulation}, the solution and test spaces, as well as the space for the truncated forcing term are not weighted. While this approach is in general not well-defined, it is for the chosen subspace under the additional assumption $\frac{a_{\max}}{a_{\min}} \in L^r(\Gamma; \R)$ for some $r > 1$.

\begin{theorem} \label{Thm:RitzEnergyFunction}
Let $\frac{a_{\max}}{a_{\min}} \in L^r(\Gamma; \R)$ for some $r > 1$, then there exists a unique solution $u^*\in \tLH$ of Problem \ref{Problem:WeightedWeakSGA}. Furthermore, the Ritz energy functional $\twE: \tLH \to \R$, defined by
\begin{align*}
\twE(\tu) := \frac{1}{2} \twB (\tu, \tu) - \twF (\tu),
\end{align*}
has a unique minimum. The corresponding minimizer 
\begin{align*}
u^* = \underset{\tu \in \tLH}{\operatorname{arg \, min}} \twE (\tu)
\end{align*}
is the unique stochastic Galerkin approximation of Problem \ref{Problem:WeightedWeakStochasticFormulation}, i.e.\ the solution of Problem \ref{Problem:WeightedWeakSGA}.
\end{theorem}
\begin{proof}
We first show $L^{2; (M)}(\Gamma; H^{-1}(\DD; \R)) \subset L^2_{a_{\min}^{-2}}(\Gamma; H^{-1}(\DD; \R))$. Let $\tf \in L^{2; (M)}(\Gamma; H^{-1}(\DD; \R))$, then $\tf (y,x) = \sum_{k=0}^M f_k(x) p_k(y)$ with $f_k \in H^{-1}(\DD; \R)$ for $k = 0, \dots , M$. It holds
\begin{align*}
\lVert \tf \rVert_{L^2_{a_{\min}^{-2}}(\Gamma; H^{-1}(\DD; \R))}^2 
	&=
\int_\Gamma \big\lVert \sum_{k=0}^M f_k(\cdot) p_k(y) \big\lVert_{H^{-1}(\DD; \R)}^2 a_{\min}^{-2}(y) \, d\mu(y)
	\\
	&\leq
\int_\Gamma \Big( \sum_{i=0}^M \lVert f_i \rVert_{H^{-1}(\DD; \R)}^2 \Big) \Big(\sum_{j = 0}^M p_j^2(y) \Big)\, a_{\min}^{-2}(y) \, d \mu(y) 
	\\
	&=
\sum_{i,j=0}^M \lVert f_i \rVert_{H^{-1}(\DD; \R)}^2 \int_\Gamma p_j^2(y) a_{\min}^{-2}(y) \, d \mu(y)
	\\
	& \leq
\sum_{i,j=0}^M \lVert f_i \rVert_{H^{-1}(\DD; \R)}^2 
\bigg( \int_\Gamma \lvert p_j(y) \rvert^{2r} \, d\mu(y) \bigg)^{1/r}
\bigg( \int_\Gamma \lvert a_{\min}^{-1}(y) \rvert^{2s} \, d\mu(y) \bigg)^{1/s} < \infty,
\end{align*}
for any $r,s > 1$ with $\frac{1}{r} + \frac{1}{s} = 1$, due to Assumption \ref{Assumption:Requirements_ExUnique_WeakSol_REPDE} and since $Y: \Omega \to \Gamma$ has finite moments of all orders. 
Analogously to Problem~\ref{Problem:WeightedWeakStochasticFormulation}, we have the existence and uniqueness of a solution $\tilde{u} \in L^2_{a/a_{\min}}(\Gamma; H_0^1(\DD; \R))$ of the truncated weighted weak formulation
\begin{align*}
\int_\Gamma \int_\DD \bar{\mathbf{a}}^{(M)}(y,x) \nabla \tilde{u}(y,x) \cdot \nabla v(y,x) \, d\lambda^d(x) \, d\mu(y) 
=
\int_\Gamma \int_\DD (\bar{a}_{\min}^{\, -1})^{(M)}(y) \tf (y, x) v(y, x)\, d\lambda^d(x) \, d\mu(y)
\end{align*}
for every $v \in L^2_{a/a_{\min}}(\Gamma; H_0^1(\DD; \R))$.
We now show $\tLH \subset L^2_{a/a_{\min}}(\Gamma; H_0^1(\DD; \R))$. 
To this end, let $\tu (y, x) = \sum_{k = 0}^M u_k(x) p_k(y) \in \tLH$, then we have
\begin{align*}
\lVert \tu \rVert_{L^2_{a/a_{\min}}(\Gamma; H_0^1(\DD; \R))}^2 
	&=
\int_\Gamma \int_\DD \frac{a(y, x)}{a_{\min}(y)} \lVert \nabla \tu(y, x) \rVert_2^2 \, d\lambda^d(x) \, d\mu(y)
	\\
	& \leq
\sum_{i,j=0}^M \int_\DD \lVert \nabla u_i(x) \rVert_2^2 \, d\lambda^d(x)
\int_\Gamma \frac{a_{\max}}{a_{\min}}(y) p_j^2(y) \, d \mu(y)
	\\
	& \leq
\sum_{i,j=0}^M \lVert u_i \rVert_{H_0^1(\DD;\R)}^2 
\bigg( \int_\Gamma \Big\lvert \frac{a_{\max}}{a_{\min}}(y) \Big\rvert^r \, d \mu(y) \bigg)^{1/r}
\bigg( \int_\Gamma \lvert p_j(y) \rvert^{2s} \, d \mu(y) \bigg)^{1/s}
	< \infty,
\end{align*}
for some $r > 1$, where $r,s > 1$ denote conjugated Hölder exponents, and since $Y: \Omega \to \Gamma$ has finite moments of all orders.
Therefore, we get the existence of a unique solution $u^* \in \tLH$ of Problem~\ref{Problem:WeightedWeakSGA}, which is quasi optimal by Céa's Lemma (\citep[Thm. 8.21]{Hackbusch2017}).
For the second statement, we show that the Ritz energy functional is bounded from below. By using the coercivity of the bilinear form $\twB$ with constant $C_1 > 0$, and the continuity of $\twF$ with constant $C_2 > 0 $, we get
\begin{align*}
\twE (\tu) 
	=
\frac{1}{2} \twB (\tu, \tu) - \twF (\tu) 
	\geq
\frac{1}{2} C_1 \lVert \tu \rVert_{\tLH}^2 - C_2 \lVert \tu \rVert_{\tLH}
	\geq
- \frac{C_2^2}{2C_1}.
\end{align*}
Since the Ritz energy functional is bounded from below, there is a minimal sequence $(u_n)_{n \in \N} \subset \tLH$, such that
\begin{align*}
\lim_{n \to \infty} \twE (u_n) = \mathcal{E}_- := \inf_{\tu \in \tLH} \twE (\tu).
\end{align*}
Using the parallelogram law in Hilbert spaces, we get for $m, n \in \N$:
\begin{align*}
\frac{C_1}{4} \lVert u_n - u_m \rVert_{\tLH}^2
	&=
\frac{C_1}{2} \lVert u_n \rVert_{\tLH}^2 + \frac{C_1}{2} \lVert u_m \rVert_{\tLH}^2 - \frac{C_1}{4} \lVert u_n + u_m \rVert_{\tLH}^2
	\\
	&=
\frac{C_1}{2} \lVert u_n \rVert_{\tLH}^2 + \frac{C_1}{2} \lVert u_m \rVert_{\tLH}^2 - C_1 \left\lVert \frac{u_n + u_n}{2} \right\rVert_{\tLH}^2
	\\
	& \qquad
		- \twF (u_n) - \twF (u_m) + 2 \twF \left( \frac{u_n + u_m}{2} \right) 
	\\
	& \leq
\frac{1}{2} \twB (u_n, u_n) - \twF (u_n) + \frac{1}{2} \twB (u_m, u_m) - \twF (u_m)
	\\
	& \qquad
 		- \twB \left( \frac{u_n + u_m}{2}, \frac{u_n + u_m}{2} \right) + 2 \twF \left( \frac{u_n + u_m}{2} \right)
 	\\
 	& =
\twE (u_n) + \twE (u_m) - 2 \twE  \left( \frac{u_n + u_m}{2} \right) 
	\\
	& \leq
\twE (u_n) + \twE (u_m)  -2 \mathcal{E}_- \to 0 \text{ for } m,n \to \infty.
\end{align*}
Therefore, $(u_n)_{n \in \N}$ is a Cauchy sequence in $\tLH$, and hence has a limit $u^* \in \tLH$. 
Due to the continuity of $\twE$, we get $\twE (u^*) = \mathcal{E}_-$, i.e.\ the infimum is attained, and thus a minimum. 
To show that the minimizer $u^*$ is the solution of Problem \ref{Problem:WeightedWeakSGA}, we define the scalar function $\phi: \R \to \R$, for any $\tv \in \tLH$, via
\begin{align*}
\phi(\lambda) 
	& :=
\twE (u^* + \lambda \tv) = \frac{1}{2} \twB (u^* + \lambda \tv, u^* + \lambda \tv) - \twF (u^* + \lambda \tv) \\
	& = 
\frac{1}{2} \twB (u^*, u^*) + \lambda \twB (u^*, \tv) + \frac{1}{2} \lambda^2 \twB (\tv, \tv) - \twF (u^*) - \lambda \twF (\tv).
\end{align*}
If $\twE$ attains its minimum at $u^*$, then $\phi$ has a minimum at $\lambda = 0$ for every $\tv \in \tLH$. 
Since $\phi$ is differentiable in $\lambda \in \R$ and parabolic with non-negative leading coefficient (due to the coercivity of $\twB$), we get the necessary condition $\phi'(0) = 0$ which is equivalent to
\begin{align*}
\twB (u^*, \tv) = \twF (\tv) \text{ for every } \tv \in \tLH.
\end{align*}
The uniqueness of the minimum of $\twE$ follows directly from the uniqueness of the solution to Problem \ref{Problem:WeightedWeakSGA}.
\end{proof}

\begin{remark} \label{Rmk:RitzEnergyClassicalSetup}
Note that Theorem \ref{Thm:RitzEnergyFunction} is also valid for the classical weak form stochastic Galerkin approximation, Problem \ref{Problem:ClassicalWeakSGA}, under the weaker assumption of $a_{\max} \in L^r(\Omega; \R)$ for some $r > 1$. The proof is analogue to the one given.
\end{remark}
If the weighted stochastic diffusion coefficient $\mathbf{a} \in L^2(\Omega; L^2(\DD; \R))$ is fully resolved by the PC expansion, e.g.\ if $\mathbf{a}$ is given by a KLE-type expansion and $P \geq 1$, there is no additional approximation error by considering the truncated operators $\twB$ and $\twF$.

\begin{corollary}
Let $\frac{a_{\max}}{a_{\min}} \in L^r(\Gamma; \R)$ for some $r > 1$, $\bar{\mathbf{a}}^{(M)} = \bar{\mathbf{a}}$ and let $\tilde{u} \in L^2_{a/a_{\min}}(\Gamma; H_0^1(\DD; \R))$ be the solution of the weighted stochastic weak formulation, Problem~\ref{Problem:WeightedWeakStochasticFormulation}, then the orthogonal projection $u^* \in \tLH$ of $\tilde{u}$ in $\tLH$ is the unique solution of the weighted weak stochastic Galerkin approximation, Problem~\ref{Problem:WeightedWeakSGA}.
\end{corollary}
\begin{proof}
As shown in the proof of Theorem~\ref{Thm:RitzEnergyFunction}, $\tLH \subset L^2_{a/a_{\min}}(\Gamma; H_0^1(\DD; \R))$, and $L^{2; (M)}(\Gamma; H^{-1}(\DD; \R)) \subset L^2_{a_{\min}^{-2}}(\Gamma; H^{-1}(\DD; \R))$. 
Now let $u^* \in \tLH$ denote the orthogonal projection of the solution $\tilde{u} \in L^2_{a/a_{\min}}(\Gamma; H_0^1(\DD; \R))$ in $\tLH$, i.e.\
\begin{align*}
\langle \tilde{u} - u^* \, , \, \tv \rangle_{L^2_{a/a_{\min}}(\Gamma; H_0^1(\DD; \R))} = 0 \text{ for all } \tv \in \tLH .
\end{align*}
By definition, we get for every $\tv \in \tLH$:
\begin{align*}
0 = \langle \tilde{u} - u^* \, , \, \tv \rangle_{L^2_{a/a_{\min}}(\Gamma; H_0^1(\DD; \R))}
	&=
\int_\Gamma \int_\DD \bar{\mathbf{a}}(y,x) \nabla (\tilde{u}(y, x) - u^*(y, x)) \cdot \nabla \tv(y, x) \, d\lambda^d(x) \, d \mu(y)
	\\
	&=
\wB (\tilde{u}, \tv ) - \wB (u^*, \tv)
	= \wF (\tv) - \wB (u^* , \tv).
\end{align*}
And hence $\wB ( u^*, \tv) = \wF(\tv)$ for every $\tv \in \tLH$. According to the assumption $\bar{\mathbf{a}}^{(M)} = \bar{\mathbf{a}}$ we get $\twB = \wB$ on $\tLH \times \tLH$, and by the orthonormality of $(p_k)_{k \in \N_0}$, we also get $\twF = \wF$ on $L^{2; (M)}(\Gamma; H^{-1}(\DD; \R))$, which concludes the proof.
\end{proof}
\subsection{Computational form of the stochastic Galerkin systems}
We assemble the systems of equations derived in the previous chapter for the given stochastic Galerkin approximations (SGA). Starting with the strong form stochastic Galerkin approximation, Problem \ref{Problem:StrongFormSGA}, we get for $k = 0, \ldots, M$:
\begin{align*}
\big\langle \mathcal{R}_{\bar{u}^{(M)}}(\cdot, x) \, , \, p_k \big\rangle_{L^2(\Gamma;\R)} 
	&= 
\Big\langle \nabla \cdot \Big( \big( \sum_{i=0}^M a_i(x) \, p_i \big) \, \nabla \big( \sum_{j = 0}^M u_j(x) \, p_j \big) \Big) + \sum_{n = 0}^M f_n(x) \, p_n \, , \, p_k \Big\rangle_{L^2(\Gamma;\R)}  
	\\
	& = 
\sum_{i,j = 0}^M \nabla \cdot (a_i(x) \nabla u_j(x)) \, \langle p_i p_j \, , \, p_k \rangle_{L^2(\Gamma; \R)} 
+ \sum_{n=0}^M f_n(x) \langle p_n \, , \, p_k \rangle_{L^2(\Gamma; \R)} 
	\\
	& = 
\sum_{i,j = 0}^M \big( (\nabla a_i(x) \cdot \nabla u_j(x) + a_i(x) \Delta u_j(x)) \, \langle p_i  p_j \, , \, p_k \rangle_{L^2(\Gamma; \R)} \big) + f_k(x) 
	&& 
x \in \DD, 
	\\
\big\langle \bar{u}^{(M)}(\cdot, x) \, , \, p_k \big\rangle_{L^2(\Gamma;\R)} 
	&= 
\Big\langle \sum_{i=0}^M u_i(x) \, p_i , \, \, p_k \Big\rangle_{L^2(\Gamma; \R)} = u_k(x)
	&& x \in \partial \DD.
\end{align*}
By defining the Galerkin tensor $G = (G_{ijk}) \in \R^{(M + 1) \times (M + 1) \times (M + 1)}$, the matrix operator $A: \DD \to \R^{(M + 1) \times (M + 1)}, \, A(x) = (A_{jk}(x))$, and the tensor operator $B: \DD \to \R^{(M + 1) \times (M + 1)
 \times d}, \, B(x) = (B_{jk}(x))$ via
\begin{align*}
G_{ijk} := \langle p_i p_j \, , \, p_k \rangle_{L^2(\Gamma;\R)} \in \R,
	\quad
A_{jk}(x) := \sum_{i = 0}^M a_i(x) \, G_{ijk} \in \R,
	\quad
B_{jk}(x) := \sum_{i = 0}^M \nabla a_i(x) \, G_{ijk} \in \R^d,
\end{align*}
the stochastic Galerkin system can be written in the following way.
\begin{problem}[Computational form of the strong form SGA] \label{Problem:ComputationalFormStrongSGA}
For the given operators defined above, find $u_0, \dots , u_M \in H_0^2(\DD; \R)$ such that
\begin{align*}
- \sum_{j = 0}^M \big( A_{jk}(x) \, \Delta u_j(x) + B_{jk}(x) \cdot \nabla u_j(x)\big) 
	&= 
f_k(x)
	&& \hspace*{-2.5cm}
 x \in \DD, \, k = 0, \dots, M,
 	\\
u_k(x) 
	&= 
0
	 && \hspace*{-2.5cm}
x \in \partial \DD, \, k = 0, \dots , M.
\end{align*}
\end{problem}
We continue with the weak form stochastic Galerkin approximation and obtain for $\tu, \tv \in \tLH$
\begin{align*}
\twB (\tu, \tv) 
	&=
\int_\Gamma \int_\DD \Big( \sum_{i = 0}^M \mathbf{a}_i(x) p_i(y) \Big) \, \nabla \Big( \sum_{j = 0}^M  u_j(x) p_j(y) \Big) \cdot \nabla \Big( \sum_{k = 0}^M v_k(x) p_k(y) \Big) \, d\lambda^d(x) \, d\mu (y)
	\\
	&=
\int_\DD \sum_{i,j,k = 0}^M G_{ijk} \mathbf{a}_i(x) \nabla u_j(x) \cdot \nabla v_k(x) \, d\lambda^d(x)
	\\
	&=
\int_\DD \sum_{j,k = 0}^M \mathbf{A}_{jk}(x) \nabla u_j(x) \cdot \nabla v_k(x) \, d\lambda^d(x),
\end{align*}
where the matrix operator $\mathbf{A} : \DD \to \R^{(M+1) \times (M+1)}$ is defined analogously by
\begin{align*}
\mathbf{A}_{jk}(x) := \sum_{i = 0}^M \mathbf{a}_i(x) G_{ijk} \in \R.
\end{align*}
Since $f \in L^2_{a_{\min}^{-2}}(\Omega; H^{-1}(\DD; \R))$, we also note that $\bar{\mathbf{f}}(y,x):= \frac{\bar{f}(y,x)}{\bar{a}_{\min}(y)} \in L^2(\Gamma; L^2(\DD; \R))$, which we analogously expand into the PC expansion to obtain the truncation
\begin{align*}
\bar{\mathbf{f}}^{(M)}(y,x) = \sum_{k = 0}^M \mathbf{f}_k(x) p_k(y).
\end{align*}
We get for $\tv \in \tLH$
\begin{align*}
\twF (\tv)
	&=
\int_{\Gamma} 
\int_\DD
\Big( \sum_{i = 0}^M \mathbf{f}_i(x) p_i(y)\Big) \Big( \sum_{j = 0}^M v_j(x) p_j(y) \Big) \, d\lambda^d(x) \, d\mu (y) 
	\\
	&=
\sum_{i,j = 0}^{M} \Big(
\int_\DD \mathbf{f}_i(x) v_j(x) \, d\lambda^d(x)
 \int_\Gamma p_i(y) p_j(y) \, d \mu (y) \Big)
	\\
	&=
\sum_{i = 0}^M \int_\DD  \mathbf{f}_i(x) v_i(x) \, d\lambda^d(x).
\end{align*}
With that, the stochastic Galerkin system can be written in the following way.
\begin{problem}[Computational form of the weighted weak form SGA] \label{Problem:ComputationalFormWeightedWeakSGA}
For the given operator defined above, find $u_0, \dots , u_M \in H_0^1(\DD; \R)$ such that for every $\tv = \sum_{i = 0}^M v_i(x) p_i(y) \in \tLH$:
\begin{align*}
\int_\DD \sum_{i,j = 0}^M \mathbf{A}_{ij}(x) \nabla u_i(x) \cdot \nabla v_j(x) \, d\lambda^d(x)
	&= 
\int_\DD \sum_{k = 0}^M \mathbf{f}_k(x) v_k(x) \, d\lambda^d(x).
\end{align*}
\end{problem}
\begin{remark} \label{Remark:EnergyMinimizationWeightedWeakSGA}
Combining Theorem \ref{Thm:RitzEnergyFunction} with the computational form, the solution $u^* \in \tLH$, of Problem \ref{Problem:ComputationalFormWeightedWeakSGA}, is given as
\begin{align*}
u^* = \underset{\tu \in \tLH}{\operatorname{arg \, min}} 
\int_\DD \frac{1}{2}  \sum_{i,j = 0}^M \mathbf{A}_{ij}(x) \nabla u_i(x) \cdot \nabla u_j(x) - \sum_{k=0}^M \mathbf{f}_k(x) u_k(x) \, d\lambda^d(x).
\end{align*}
\end{remark}
The computational forms are the basis for the training strategies of the neural networks, developed in the next section.

\section{Deep learning approach} \label{sec:DL}
We use neural networks as surrogate for the spectral coefficients of the stochastic Galerkin approximations defined in the previous sections. 
A theoretical justification for that is given by the so-called \textit{universal approximation property}, stating that feedforward neural networks (in different setups) can approximate any continuous function on a bounded domain up to arbitrary precision. Here we mention the approximation property for sigmoidal and more general activation functions, as used in this work (see e.g.\ \citep{Cybenko1989, Hornik1989, Mhaskar1993}).
\subsection{Definition and notation}
In this section we introduce the notation and basic setup of the neural networks used in this work. In this context we approximate the spectral coefficients $(u_0, \dots, u_M)^T: \overline{\DD} \to \R^{M+1}$ of the stochastic Galerkin approximation $\tu (y,x) = \sum_{i = 0}^M u_i(x) p_i(y)$ of Problem \ref{Problem:ComputationalFormStrongSGA} and \ref{Problem:ComputationalFormWeightedWeakSGA} respectively by deep feedforward neural networks. 

A deep feedforward neural network of depth $D \in \N$ defines a parameterized mapping $\mathcal{N}_{\theta} : \R^{d_0} \to \R^{d_{D+1}}$, with parameters $\theta \in \Theta$, that takes an input $x \in \R^{d_0}$, and consists of $D$ hidden layers $\mathcal{L}^{(1)}, \dots , \mathcal{L}^{(D)}$ of sizes $d_1, \dots , d_D \in \N$, and an output layer $\mathcal{L}^{(D+1)}$ of size $d_{D+1}$. 
The size of a layer determines the number of computational nodes, called neurons, it consists of. Each neuron computes an affine transformation of the output of the previous layer, respectively the input, which is then composed, called activated, with a usually non-linear function called activation function. Formally, the output $l^{(i)} \in \R^{d_i}$ of the $i$-th layer $\mathcal{L}^{(i)} : \R^{d_{i-1}} \to \R^{d_i}$, for $i = 1, \dots, D+1$, is given by
\begin{align*}
l^{(i)} = \mathcal{L}^{(i)}(l^{(i-1)}) := \sigma^{(i)} \left( W^{(i)} l^{(i-1)} + b^{(i)} \right),
\end{align*}
where $\sigma^{(i)}: \R \to \R$ is the activation function of the $i$-th layer, applied component-wise, $W^{(i)} \in \R^{d_i \times d_{i-1}}$ and $b^{(i)} \in \R^{d_i}$ denote the weight matrix and the bias of the $i$-th layer, and $l^{(i-1)} \in \R^{d_{i-1}}$ the output of the $(i-1)$-th layer, respectively the input $x \in \R^{d_0}$ to the neural network for $i - 1 = 0$. The $j$-th component of $l^{(i)} = (l_1^{(i)} , \dots , l_{d_i}^{(i)})$ is the output of the $j$-th neuron in the $i$-th layer.
The deep feedforward neural network is defined as the composition of the layers
\begin{align*}
\mathcal{N}_{\theta} : \R^{d_0} \to \R^{d_{D+1}}, \, x \mapsto \left( \mathcal{L}^{(D+1)} \circ \mathcal{L}^{(D)} \circ \dots \circ \mathcal{L}^{(1)} \right)(x),
\end{align*}
where the parameters, called weights, are given by
\begin{align*}
\theta = \left( W^{(1)}, b^{(1)}, \dots , W^{(D+1)}, b^{(D+1)} \right) \in \Theta \subset \R^{d_1 \times d_0} \times \R^{d_1} \times \dots \times \R^{d_{D+1} \times d_D} \times \R^{d_{D+1}} := \mathcal{R}.
\end{align*}
The parameter space $\Theta$ can be a proper subset of $\mathcal{R}$, placing additional constraints on the weights, e.g.\ non-negativity or special weight matrices resulting in disconnected neurons. If smooth activation functions are used, the neural network is itself a smooth function.
Neural networks have a very high expressive power and achieve an impressive performance in a wide range of tasks. The task is usually taught to the neural network by a specified loss function
penalizing wrong predictions in a proper way. 
The algorithm then seeks to find a set of weights that minimizes some norm or transformation of the loss function, called risk. In supervised learning, the optimization is carried out with the help of a dataset containing ground-truth data, whereas in unsupervised learning no such dataset is used, the loss is calculated directly via a transformation of the input, without the need of ground truth labels. This can be advantageous, when training labels are very expensive or not available. Combinations of both training methods are also very common.

\subsection{Neural network architecture and loss function}
In this work we introduce two neural networks, the \textit{S-GalerkinNet}, approximating a solution of the strong form stochastic Galerkin method, Problem \ref{Problem:StrongFormSGA}, and the \textit{S-RitzNet}, approximating a solution of the weighted weak form stochastic Galerkin method, Problem \ref{Problem:WeightedWeakSGA}. 
The training strategies, namely the corresponding loss functions, are derived from the computational forms, Problem \ref{Problem:ComputationalFormStrongSGA}, and Problem \ref{Problem:ComputationalFormWeightedWeakSGA} using the minimization described in Remark \ref{Remark:EnergyMinimizationWeightedWeakSGA}. Furthermore, a schematic overview of the architectures of the neural networks is given. 

Starting with the \textit{S-GalerkinNet}, the goal is to  approximate the spectral coefficients $u_0, \dots , u_M \in H_0^2(\DD; \R)$ of the stochastic Galerkin approximation $\tu (y,x) = \sum_{i = 0}^M u_i(x) p_i(y) \in L^{2; (M)}(\Gamma; H_0^2(\DD; \R))$, solving Problem \ref{Problem:StrongFormSGA}, by the neural network.
In fact, we don't approximate the spectral coefficients $(u_0, \dots , u_M)^T: \overline{\DD} \to \R^{M+1}$ directly by the \textit{S-GalerkinNet} $\mathcal{N}_\theta^{\operatorname{SG}} : \overline{\DD} \to \R^{M+1}$, but we incorporate the homogeneous Dirichlet boundary conditions as a hard constraint. For that, we use an enforcer function $e \in C^2(\overline{\DD}; \R)$, such that $e(x) = 0$ for $x \in \partial \DD$, and $e(x) > 0$ for $x \in \DD$.
With that, we define the function
\begin{align*}
{\scriptstyle \mathcal{U}}_{\theta^*}^{\operatorname{SG}} = ({\scriptstyle \mathcal{U}}_{0;\theta^*}^{\operatorname{SG}}, \dots , {\scriptstyle \mathcal{U}}_{M;\theta^*}^{\operatorname{SG}})^T: \overline{\DD} \to \R^{M+1}, \, x \mapsto e(x) \cdot \mathcal{N}_{\theta^*}^{\operatorname{SG}}(x),
\end{align*}
and use the components as approximation of the spectral coefficients of the strong form stochastic Galerkin approximation $\tu \in L^{2;(M)}(\Gamma; H_0^2(\DD; \R))$, i.e.\
\begin{align*}
\tu (y,x) \approx \sum_{i = 0}^M {\scriptstyle \mathcal{U}}_{i;\theta^*}^{\operatorname{SG}}(x) p_i(y) =: \tu_{\theta^*}(y,x).
\end{align*}
Note that enforcing the boundary conditions as a hard constraint is not limited to homogeneous Dirichlet boundary conditions or rectangular domains (see e.g.\ \citep{Berrone2022, Sukumar2021}). 
To derive the optimization scheme to find the optimal parameters $\theta^* \in \Theta^{\operatorname{SG}}$, for the above defined approximation, we recall the goal to find ${\scriptstyle \mathcal{U}}_{\theta^*}^{\operatorname{SG}}$ such that for every $k = 0 , \dots , M$ and $x \in \DD$
\begin{align*}
\big\langle \mathcal{R}_{\tu_{\theta^*}}(\cdot, x) , p_k \big\rangle_{L^2(\Gamma; \R)} \overset{!}{=} 0. 
\end{align*}
In order to define a proper loss function to achieve this goal, we note the following:
\begin{theorem}
For smooth activation functions and every $\theta \in \Theta^{\operatorname{SG}}$ and $k = 0, \dots , M$, the mapping
\begin{align*}
x \mapsto \big\langle \mathcal{R}_{\tu_{\theta}}(\cdot, x) , p_k \big\rangle_{L^2(\Gamma; \R)}
\end{align*}
is a member of $L^2(\DD; \R)$.
\end{theorem}
\begin{proof}
We get for every $\theta \in \Theta^{\operatorname{SG}}$ and $k = 0, \dots , M$
\begin{align*}
& \int_\DD \big\lvert
 \big\langle \mathcal{R}_{\tu_{\theta}}(\cdot, x) , p_k \big\rangle_{L^2(\Gamma; \R)} 
 \big\rvert^2 \, d\lambda^d(x)
 	\\
 	& \hspace*{0.5cm}
 	=
\int_\DD \Big\lvert 
\sum_{i,j = 0}^M \big(
\nabla a_i(x) \cdot \nabla {\scriptstyle \mathcal{U}}_{j;\theta}^{\operatorname{SG}}(x) + a_i(x) \Delta {\scriptstyle \mathcal{U}}_{j;\theta}^{\operatorname{SG}}(x)
\big) \, \langle p_i  p_j \, , \, p_k \rangle_{L^2(\Gamma; \R)} 
+ f_k(x) 
\big\rvert^2 \, d\lambda^d(x)
	\\
	& \hspace*{0.5cm}
	\leq
2(M+1)^2 \sum_{i,j=0}^M 
\int_\DD
\big\lvert
\big( 
 	\nabla a_i(x) \cdot \nabla {\scriptstyle \mathcal{U}}_{j;\theta}^{\operatorname{SG}}(x) + a_i(x) \Delta {\scriptstyle \mathcal{U}}_{j;\theta}^{\operatorname{SG}}(x)
 \big)
\langle p_i p_j \, , \, p_k \rangle_{L^2(\Gamma; \R)}
\big\rvert^2 \, d\lambda^d(x)
+ 2 \int_\DD \lvert f_k \rvert^2 \, d\lambda^d(x)
	\\
	& \hspace*{0.5cm}
	\leq
4(M+1)^2 \sum_{i,j=0}^M \lVert p_i p_j \rVert_{L^2(\Gamma; \R)}^2
\Big(
\int_\DD 
\big\lVert \nabla a_i(x) \big\rVert_2^2 \, \big\lVert  \nabla {\scriptstyle \mathcal{U}}_{j;\theta}^{\operatorname{SG}}(x) \big\rVert_2^2 
+
\big\lvert
a_i(x) \Delta {\scriptstyle \mathcal{U}}_{j;\theta}^{\operatorname{SG}}(x)
 \big\rvert^2
\, d\lambda^d(x)
\Big)
+ 2 \lVert f_k \rVert_{L^2(\DD, \R)}^2
	\\
	& \hspace*{0.5cm}
	\leq
4(M+1)^2 \sum_{i,j=0}^M 
C_j \lVert p_i p_j \rVert_{L^2(\Gamma; \R)}^2
\, \lVert a_i \rVert_{H^1(\DD; \R)}^2 
+ 2 \lVert f_k \rVert_{L^2(\DD, \R)}^2
< \infty,
\end{align*}
where
\begin{align*}
C_j := \max
\Big\{
\max_{x \in \overline{\DD}} \lVert  \nabla {\scriptstyle \mathcal{U}}_{j;\theta}^{\operatorname{SG}}(x) \rVert_2^2 
	\, , \,
\max_{x \in \overline{\DD}} \lvert \Delta {\scriptstyle \mathcal{U}}_{j;\theta}^{\operatorname{SG}}(x) \rvert^2
\Big\},
\end{align*}
which is finite due to Assumption \ref{Assumption:additional_assm_strong_residual}, and since the neural network is a smooth function on the compact domain $\overline{\DD}$, and hence we get ${\scriptstyle \mathcal{U}}_{\theta}^{\operatorname{SG}} = e \cdot \mathcal{N}_\theta^{\operatorname{SG}} \in C^2(\overline{\DD} ; \R^{M+1})$.
\end{proof}

With that, we can define an unsupervised loss function of least-squares-type to train the neural network, namely
\begin{align*}
\mathbf{L}^{\operatorname{SG}}(x; \mathcal{N}_\theta^{\operatorname{SG}}) := \frac{1}{M + 1} \sum_{k=0}^M \big\langle \mathcal{R}_{\tu_{\theta}}(\cdot, x) , p_k \big\rangle_{L^2(\Gamma; \R)}^2.
\end{align*}
Furthermore, we work with the computational form of the strong form stochastic Galerkin approximation, Problem \ref{Problem:ComputationalFormStrongSGA}. The associated risk is the corresponding empirical $L^2(\DD;\R)$ norm of the above defined loss, i.e. the training procedure of the \textit{S-GalerkinNet} is given as follows:
\begin{problem}[Training of the \textit{S-GalerkinNet}] \label{Problem:Training_SGN}
For given training points $(x_1, \dots , x_n) \in \DD^n$, find a set of parameters $\theta^* \in \Theta^{\operatorname{SG}}$, such that
\begin{align*}
\lambda^d(\DD) \, \frac{1}{n} \sum_{l = 1}^n \mathbf{L}^{\operatorname{SG}}(x_l; \mathcal{N}_{\theta^*}^{\operatorname{SG}})
= \lambda^d(\DD) \, \frac{1}{n} \sum_{l = 1}^n \frac{1}{M+1} \sum_{k=0}^M 
\left( \sum_{j=0}^M \big(A_{jk}(x_l) \, \Delta {\scriptstyle \mathcal{U}}_{j;\theta^*}^{\operatorname{SG}}(x_l) + B_{jk}(x_l) \cdot \nabla {\scriptstyle \mathcal{U}}_{j;\theta^*}^{\operatorname{SG}}(x_l) \big) + f_k(x_l) \right)^2 \overset{!}{=} 0.
\end{align*}
\end{problem}
Working with the \textit{S-RitzNet} $\mathcal{N}_\theta^{\operatorname{SR}}: \overline{\DD} \to \R^{M+1}$, we also approximate the spectral coefficients $u_0 , \dots , u_M \in H_0^1(\DD;\R)$, of the stochastic Galerkin approximation $\tu \in \tLH$, by the neural network. In this approach, however, we work with the weighted weak form stochastic Galerkin approximation, more precisely with the energy minimization of the computational form given in Remark \ref{Remark:EnergyMinimizationWeightedWeakSGA}.
Since this defines a minimization of an integral operator over the spatial domain, choosing the integrand as loss function is quite natural for this task. Therefore, using the computational form, we define the loss function for the \textit{S-RitzNet} as follows:
\begin{align*}
\mathbf{L}^{\operatorname{SR}}(x; \mathcal{N}_\theta^{\operatorname{SR}}) := \frac{1}{2} \sum_{i,j = 0}^M \mathbf{A}_{ij}(x) \nabla {\scriptstyle \mathcal{U}}_{i ; \theta}^{\operatorname{SR}} (x) \cdot \nabla {\scriptstyle \mathcal{U}}_{j ; \theta}^{\operatorname{SR}} (x) - \sum_{k = 0}^M \mathbf{f}_k(x) {\scriptstyle \mathcal{U}}_{k ; \theta}^{\operatorname{SR}} (x),
\end{align*}
where 
\begin{align*}
{\scriptstyle \mathcal{U}}_{\theta}^{\operatorname{SR}} = ({\scriptstyle \mathcal{U}}_{0;\theta}^{\operatorname{SR}}, \dots , {\scriptstyle \mathcal{U}}_{M;\theta}^{\operatorname{SR}})^T: \overline{\DD} \to \R^{M+1}, \, x \mapsto e(x) \cdot \mathcal{N}_{\theta}^{\operatorname{SR}}(x),
\end{align*}
denotes the neural network with enforced homogeneous Dirichlet boundary conditions, as described above. The associated risk of the \textit{S-RitzNet} is then the Monte Carlo integration of the loss function, i.e.\ the training procedure is given as follows:
	\begin{figure}[t!]
	\begin{center}
	\includegraphics[width=0.9999\textwidth]{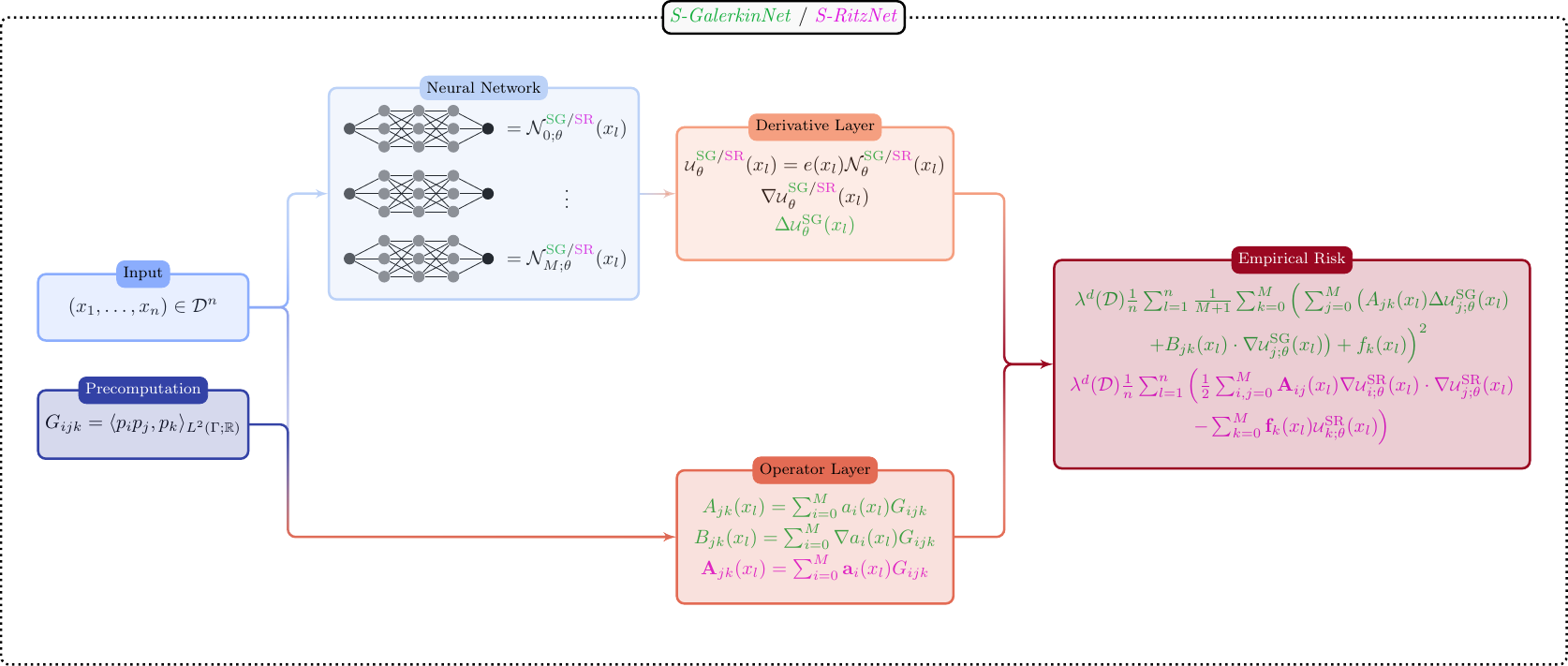}
	\end{center}
	\caption{Architecture of \textit{S-GalerkinNet} and \textit{S-RitzNet}.} \label{Fig:NN_Architecture_Combined}
	\end{figure}
\begin{problem}[Training of the \textit{S-RitzNet}] \label{Problem:TrainingSRN}
For given training points $(x_1, \dots , x_n) \in \DD^n$, find a set of parameters $\theta^* \in \Theta^{\operatorname{SR}}$, such that
\begin{align*}
&\lambda^d(\DD) \frac{1}{n} \sum_{l = 1}^n \mathbf{L}^{\operatorname{SR}}(x_l, \mathcal{N}_{\theta^*}^{\operatorname{SR}}) 
	=
\lambda^d(\DD) \frac{1}{n} \sum_{l = 1}^n 
	\Big( 
		\frac{1}{2} \sum_{i,j = 0}^M \mathbf{A}_{ij}(x_l) \nabla {\scriptstyle \mathcal{U}}_{i ; \theta^*}^{\operatorname{SR}} (x_l) \cdot \nabla {\scriptstyle \mathcal{U}}_{j ; \theta^*}^{\operatorname{SR}} (x_l) - \sum_{k = 0}^M \mathbf{f}_k(x_l) {\scriptstyle \mathcal{U}}_{k ; \theta^*}^{\operatorname{SR}} (x_l)
	\Big)
	\\
	& \qquad =
\min_{\theta \in \Theta^{\operatorname{SR}}} \lambda^d(\DD) \frac{1}{n} \sum_{l = 1}^n \mathbf{L}^{\operatorname{SR}}(x_l, \mathcal{N}_{\theta}^{\operatorname{SR}}).
\end{align*}
\end{problem}
Since the actual minimum is not known, the value of the risk is not as interpretable as for the strong residual. Therefore, every few epochs, we calculate a validation error to get a clearer quantification of the risk, and to formulate a stopping criterion for the neural network optimization. This validation error is the mean squared error of $10^4$ pathwise strong residuals.

A schematic overview of the architectures of the \textit{S-GalerkinNet} and the \textit{S-RitzNet} is given in Figure \ref{Fig:NN_Architecture_Combined}. The nodes written in black are the same for both architectures (with different weights), while the ones depicted in green only refer to the \textit{S-GalerkinNet}, and the ones depicted in magenta only refer to the \textit{S-RitzNet}. 
The input spatial points $(x_1, \dots , x_n) \in \DD^n$ are drawn in mini batches from a continuous stream of random or quasi-random samples, e.g.\ from a uniform distribution or as a Sobol sequence (see \citep{Sobol1967}).
The stochastic integrals, i.e.\ the Galerkin tensor $G \in \R^{(M+1) \times (M+1) \times (M+1)}$, are precomputed once in the offline phase.
The neural network consists of $(M+1)$ disconnected deep feedforward branches, where the $i$-th branch $\mathcal{N}_{i;\theta^*}^{\text{SG}}$ is used for the approximation of the $i$-th spectral coefficient $u_i$ of the stochastic Galerkin approximation $\tu$, for $i = 0 , \dots , M$. The specific details of the neural network, such as depth, width, activation functions, etc.\ are determined experimentally, and stated in the corresponding experiment section.
In the derivative layer, which is non-trainable, the boundary conditions are enforced, as described above, and the corresponding derivatives are calculated using automatic differentiation (see \citep{Baydin2017}). 
In the also non-trainable operator layer, the operators for the computational form are assembled, and the derivatives of the diffusion coefficient are either obtained analytically or via automatic differentiation. 
In the last forward step, the above derived empirical risk is calculated. 
In the not depicted backward step, this risk is minimized and the neural network weights are updated, called training. The training of the \textit{S-GalerkinNet} and \textit{S-RitzNet} is purely unsupervised, i.e.\ no labeled ground-truth data is used. The optimization is  based on a variant of stochastic gradient descent, called ADAM algorithm (see \citep{Kingma2014}), where gradients are calculated using the backpropagation algorithm. 
We refer to \citep[Ch. 6]{Goodfellow2016} for details on feedforward neural networks and backpropagation. 
The training procedure for the \textit{S-GalerkinNet} is stopped if the risk is below a threshold, or if it stopped decreasing for specified number of epochs. The training procedure for the \textit{S-RitzNet} is stopped if the risk and the validation error stopped decreasing for a specified number of epochs.

\section{Numerical Experiments}\label{sec:num}
In this section, we present numerical examples, where we solve the elliptic random PDE for different stochastic diffusion coefficients and forcing terms with the given deep learning approaches. The neural networks are implemented using the Python library Tensorflow (\citep{tensorflow2015-whitepaper}), and trained on a single \textit{NVIDIA RTX 3070} GPU.
The error $\epsilon_{M, \theta}$, stated in the experiments, is an approximation of the relative $L^2(\Omega; H^1(\DD; \R))$ distance of a reference solution $u$ (which is analytically available or obtained by other methods, stated in the experiments), and the neural network stochastic Galerkin approximation $u_{\theta}^{(M)} (\omega, x) = \sum_{i = 0}^M {\scriptstyle \mathcal{U}}_{i ; \theta}^{\operatorname{SG / SR}}(x) p_i(Y(\omega))$, given by the \textit{S-GalerkinNet} respectively the \textit{S-RitzNet}. The error is given by
\begin{align*}
\epsilon_{M, \theta} \approx \frac{\big\lVert u - u_{\theta}^{(M)} \big\rVert_{L^2(\Omega; H^1(\DD; \R))}}{
\lVert u \rVert_{L^2(\Omega; H^1(\DD; \R))}
} 
=
\frac{
\Big( \int_\Omega \int_\DD \big( u(\omega, x) - u_{\theta}^{(M)} (\omega, x) \big)^2 + \big\lVert \nabla u(\omega, x) - \nabla u_{\theta}^{(M)} (\omega, x) \big\rVert_2^2 \, d\lambda^d(x) \, dP(\omega) \Big)^{1/2}
}
{
\Big( \int_\Omega \int_\DD u^2 + \lVert \nabla u \rVert_2^2 \, d \lambda^d(x) \, dP(\omega) \Big)^{1/2}
},
\end{align*}
where the derivatives are calculated analytically, respectively via automatic differentiation, the spatial integrals are approximated using the trapezoid rule, and the stochastic integrals are approximated using Monte Carlo integration. 
Throughout the experiments, we see that both networks perform very similar in terms of approximation errors, while the \textit{S-RitzNet} trains substantially faster. This is because only first derivatives of \textit{S-RitzNet} are calculated during training, whereas second derivatives of the neural network are calculated during the training of the \textit{S-GalerkinNet}. The difference in training time grows with the dimensionality of the systems. 
Additionally, we measure the coupling strength of the systems as relative number of non-negative, non-diagonal entries of the coupling operators (as a mean over spatial points), $\operatorname{nnz}(A(x))$ and $\operatorname{nnz}(B(x))$ for the \textit{S-GalerkinNet}, and $\operatorname{nnz}(\mathbf{A}(x))$ for the \textit{S-RitzNet}.
The measured training times in the experiments all include the building and initialization of the tensorflow computational graph.

In a first example we consider the elliptic random PDE with constant diffusion coefficient, and a stochastic forcing term, that requires a very high polynomial degree in the PC expansion for a sufficient approximation. This results in a high dimensional system that can be accurately solved by the proposed neural network approaches.
To this end,
let $\DD = (0,1)$, and let $f(\omega, x) = \lvert \xi(\omega) - 1 \rvert$ denote the stochastic forcing term, where $\xi \sim \mathcal{N}(0,1)$. 
We use the probabilist's Hermite polynomials $(\operatorname{He}_k)_{k = 0 , \dots , M}$, evaluated at a standard normal random variable $Y \sim \mathcal{N}(0,1)$ (i.e.\ $N = 1$ and $M = P$), as polynomial chaos basis of $L^2(\R; \R)$ with respect to the standard normal distribution. The solution $u: \Omega \times [0,1] \to \R$ of the considered equation
\begin{align*}
- \frac{d^2}{dx^2} u(\omega, x) 
	&= 
	\lvert \xi(\omega) - 1 \rvert
	&& 
\hspace*{-3cm} x \in (0,1), \, \omega \in \Omega,	
	\\
u(\omega, 0) = u(\omega, 1) 
	&=
0
	&&
\hspace*{-3cm} \omega \in \Omega,
\end{align*}
is unique and analytically available, given by $u(\omega, x) = \frac{1}{2} \lvert \xi(\omega) - 1 \rvert (x - x^2)$. We choose the enforcer function $e(x):= x(1-x)$, and hence we have
\begin{align*}
{\scriptstyle \mathcal{U}}_{\theta}^{\operatorname{SG}}(x) = x(1-x) \mathcal{N}_\theta^{\operatorname {SG}}(x),
\quad
{\scriptstyle \mathcal{U}}_{\theta}^{\operatorname{SR}}(x) = x(1-x) \mathcal{N}_\theta^{\operatorname {SR}}(x).
\end{align*}
The spectral coefficients $(f_0, \dots , f_M)$ of the deterministic representation $\tf \in L^2_M(\Gamma; L^2(\DD; \R))$ are given by
\begin{align*}
f_i = \langle f , \operatorname{He}_i \rangle_{L^2(\Gamma; \R)}
= \frac{1}{\sqrt{2 \pi}} \int_{- \infty}^{\infty} \lvert y - 1 \rvert \operatorname{He}_i(y) \exp(- y^2/2) \, dy , \quad i = 0, \dots , M.
\end{align*}

Since the diffusion coefficient is constant, the loss function for the \textit{S-GalerkinNet} and for the \textit{S-RitzNet} simplify to
\begin{align*}
\mathbf{L}^{\text{SG}}(x; \mathcal{N}_\theta^{\operatorname {SG}}) = \frac{1}{M + 1} \sum_{k = 0}^M \Big( \frac{d^2}{dx^2} {\scriptstyle \mathcal{U}}_{k ; \theta}^{\operatorname{SG}} (x) + f_k \Big)^2,
\quad
\mathbf{L}^{\operatorname{SR}}(x; \mathcal{N}_{\theta}^{\operatorname{SR}}) = \sum_{k = 0}^M \frac{1}{2} \Big( \frac{d}{dx} {\scriptstyle \mathcal{U}}_{k ; \theta}^{\operatorname{SR}}(x) \Big)^2  - f_k {\scriptstyle \mathcal{U}}_{k ; \theta}^{\operatorname{SR}} (x),
\end{align*}
where the spectral coefficients $(f_0, \dots , f_M) \in \R^{M+1}$ are, due to the spatial independence, precomputed in the offline phase. 
\begin{figure}
\begin{minipage}{.5\textwidth}
  \begin{center}
  \includegraphics[width=0.99\textwidth]{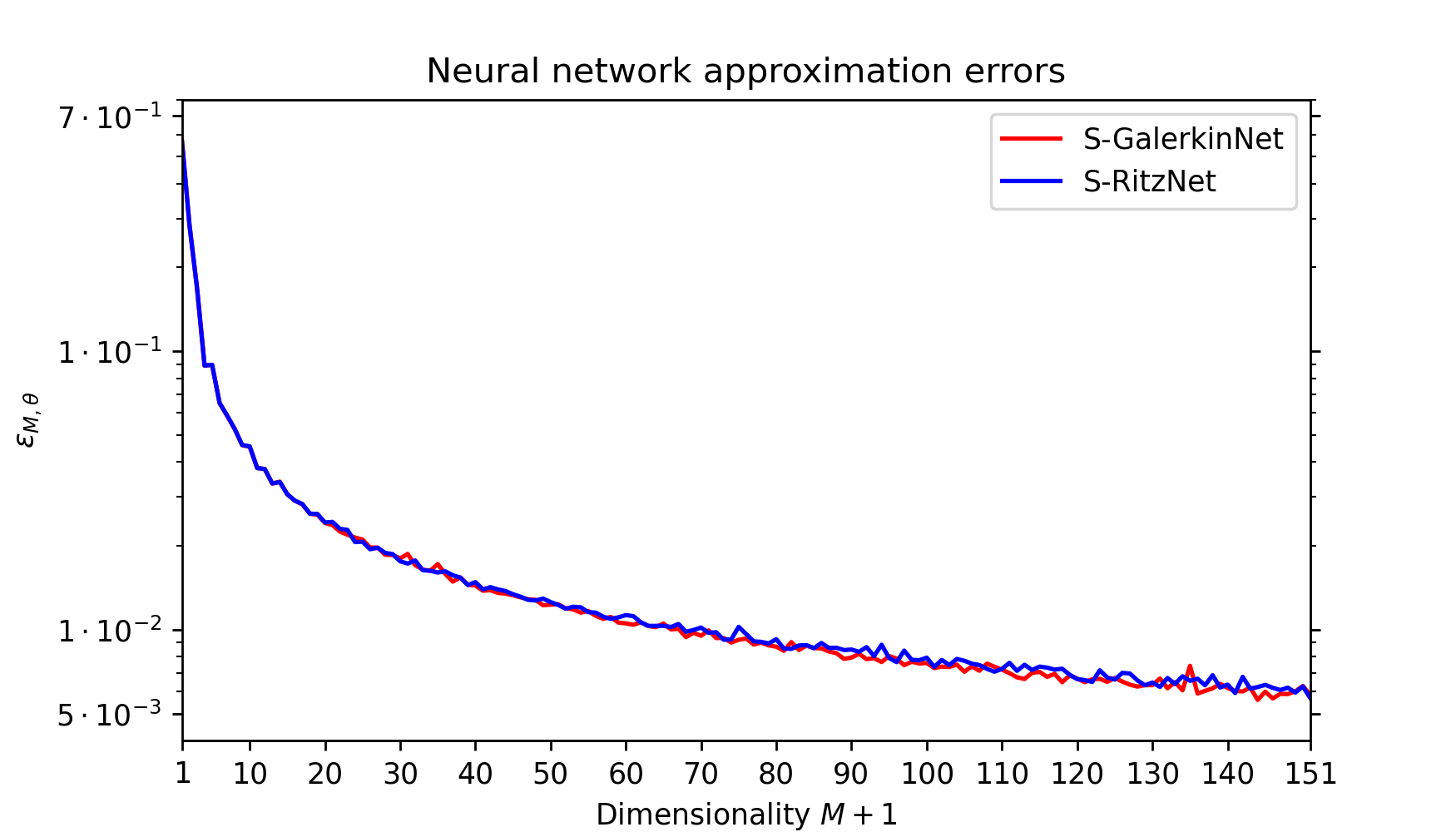} 
  \end{center}
\end{minipage}
\begin{minipage}{.5\textwidth}
  \begin{center}
  \includegraphics[width=0.99\textwidth]{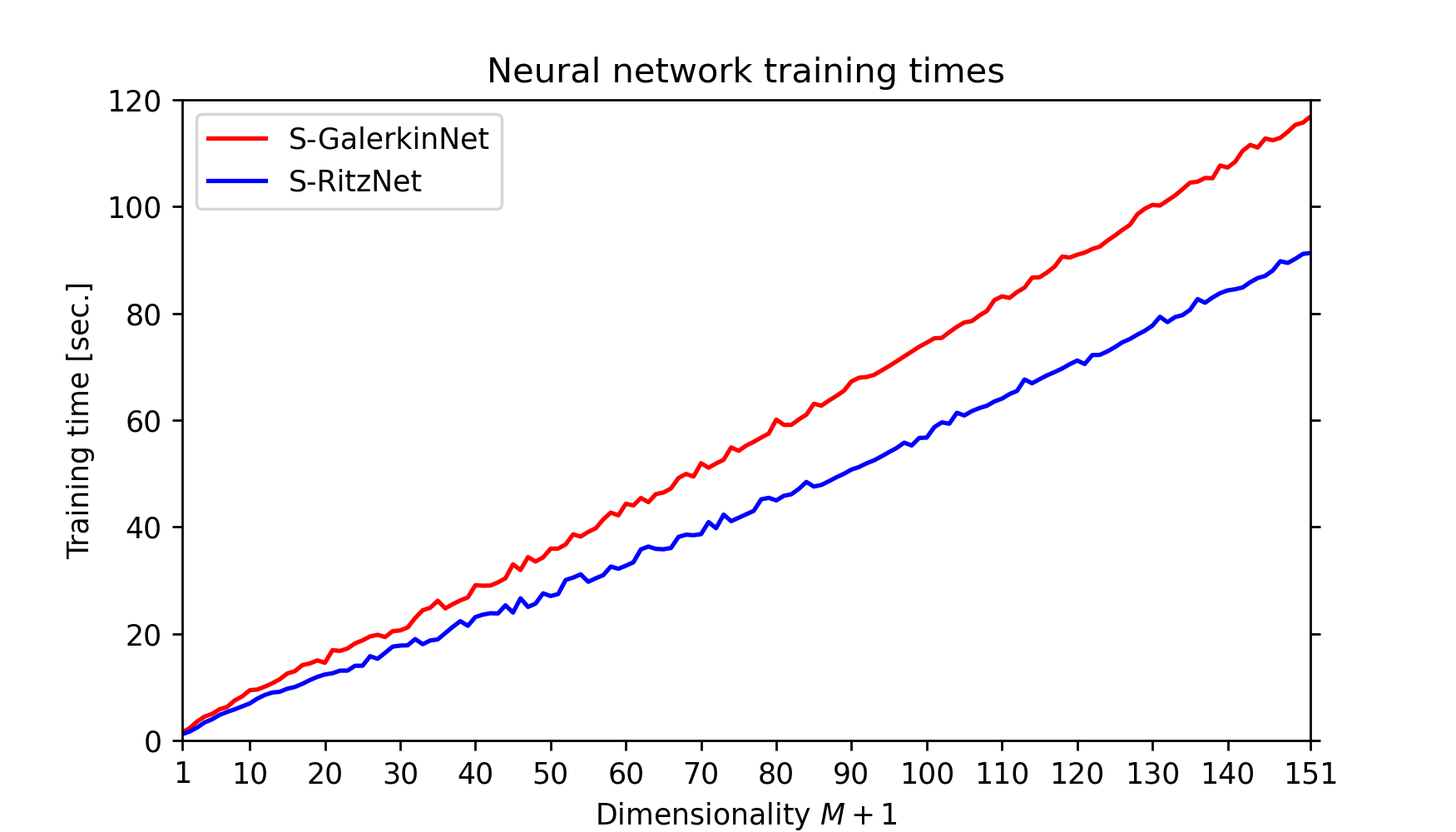}
  \end{center}
\end{minipage}
\caption{Performance and training times of \textit{S-GalerkinNet} and \textit{S-RitzNet}.}
\label{Figure:Experiment1_Data}
\end{figure}
We use the same architecture for the \textit{S-GalerkinNet} and the \textit{S-RitzNet}. Each neural network consists of $M+1$ disconnected branches, where each branch is made up by an input layer with 45 neurons and the \textnormal{swish} activation function (see \citep{Ramachandran2017}), three hidden layers, each of which has 45 \textnormal{sigmoid} activated neurons, and a linear output neuron. 
The \textnormal{ADAM} optimizer uses a decaying learning rate scheme, and the input spatial points are drawn as a continuous stream of quasi-random Sobol points.
The approximation errors as well as the training times of the neural networks are given in Figure \ref{Figure:Experiment1_Data}, for $P = 0,  \dots, 150$, resulting in systems of dimensions $M = 1, \dots , 151$. We achieve a relative $L^2(\Omega; H^1(\DD; \R))$-error of less than $0.6\%$ with the highest polynomial degree of $P = 150$ for both deep learning approaches. The resulting systems are decoupled and the training time grows linearly with rising polynomial dimension.

As a second example, we consider the elliptic random PDE in two spatial dimensions with a constant forcing term, and Karhunen-Loève-type stochastic diffusion coefficient. The diffusion coefficient is chosen in such a way that many KL terms are needed to accurately resolve the random field. This results in high dimensional systems of fully coupled two dimensional PDEs, that can be well approximated by the neural networks in reasonable training times.
Let $\DD = (0, 1)^2$, and let
\begin{align*}
a(\omega, x) = 3 - x_1 x_2 (1 - x_1) (1 - x_2) - \frac{1}{2} 
\sum_{k = 1}^\infty \left(\frac{1}{k} \right)^{8/5} e^{- \frac{1}{k} (x_1 - x_2)^2} (Y_k(\omega) + 1), \quad f(\omega, x) = 1,
\end{align*}
denote the stochastic diffusion coefficient and forcing term, where $x = (x_1, x_2)$, and $Y_k \sim \mathcal{U}([-1,1])$ independent for $k \in \N$. A truncation of the expansion after $N \in \N$ summands results in a diffusion coefficient that is driven by $Y = (Y_1, \dots , Y_N): \Omega \to \Gamma = [-1,1]^N$ satisfying Assumption~\ref{Assumption:RVsReq_gPC_MomentsAndDeterminateMeasure}. 
We use the tensor product Legendre Polynomials $(\operatorname{Le}_\nu)_{\nu \in \mathcal{I}_{N,1}}$ of maximal polynomial degree $P = 1$, evaluated in $Y$, as basis of $L^2(\Gamma; \R)$ with respect to the uniform distribution. The stochastic Galerkin approximation yields a fully coupled system of $M + 1 = N + 1$ two dimensional PDEs. 
The stochastic diffusion coefficient satisfies Assumption~\ref{Assumption:Requirements_ExUnique_WeakSol_REPDE} for deterministic bounds
\begin{align*}
0.65 = a_{\min} \leq a(\omega, x) \leq a_{\max} = 3 \text{ for every } x \in (0,1)^2, \omega \in \Omega.
\end{align*}
Hence there exists a unique solution $u \in L^2(\Omega; H_0^1(\DD; \R))$ of the classical stochastic weak formulation (Problem~\ref{Problem:ClassicalWeakStochasticFormulation}) of the considered elliptic random PDE
\begin{align*}
- \nabla \cdot \left(  \bigg(3 - x_1 x_2 (1 - x_1) (1 - x_2) - \frac{1}{2} 
\sum_{k = 1}^N \left(\frac{1}{k} \right)^{8/5} e^{- \frac{1}{k} (x_1 - x_2)^2} (Y_k(\omega) + 1) \bigg) \nabla u(\omega, x) \right)
	&=
1
	&& 
\hspace*{0.5cm} x \in (0,1)^2, \, \omega \in \Omega,	
	\\
u(\omega, x) 
	&=
0
	&&
\hspace*{0.5cm} x \in \partial \big([0,1]^2\big), \, \omega \in \Omega.
\end{align*}
Since $a_{\max} \in L^r(\Omega; \R)$ for every $r \in (0, \infty)$, by Theorem~\ref{Thm:RitzEnergyFunction}, and Remark~\ref{Rmk:RitzEnergyClassicalSetup}, there exists a unique solution $\bar{u}^{(M)} \in L^{2;(M)}(\Gamma; H_0^1(\DD; \R))$ of the classical weak form stochastic Galerkin approximation, Problem~\ref{Problem:ClassicalWeakSGA}, with a truncated KL-type expansion of the diffusion coefficient. We approximate the spectral coefficients of this solution by neural networks with the enforcer function $e(x) := x_1 x_2 (1 - x_1) (1 - x_2)$ to obtain
\begin{align*}
{\scriptstyle \mathcal{U}}_{\theta}^{\operatorname{SG}}(x) = x_1 x_2 (1 - x_1) (1 - x_2) \mathcal{N}_\theta^{\operatorname {SG}}(x),
\quad
{\scriptstyle \mathcal{U}}_{\theta}^{\operatorname{SR}}(x) = x_1 x_2 (1 - x_1) (1 - x_2) \mathcal{N}_\theta^{\operatorname {SR}}(x).
\end{align*}
The spectral coefficients $(a_0, \dots, a_M)$ of the PC expansion of the deterministic representation $\bar{a}^{(M)}: \Gamma \times \DD \to \R$, with $M = N$, sorted by the graded lexicographic ordering, are directly given by
\begin{align*}
a_0(x) 
	&=
3 - x_1 x_2 (1 - x_1) (1 - x_2) - \frac{1}{2} 
\sum_{k = 1}^{N} \left(\frac{1}{k} \right)^{8/5} e^{- \frac{1}{k} (x_1 - x_2)^2},
	\\
a_k(x)
	&=
\frac{1}{2} \left(\frac{1}{N - (k - 1)}\right)^{8/5} e^{- \frac{1}{N - (k-1)} (x_1 - x_2)^2}, \text{ for } k = 1, \dots, N.
\end{align*}
The corresponding loss functions are given in Problem~\ref{Problem:Training_SGN} and Problem~\ref{Problem:TrainingSRN} with unweighted operators.
We use the same architecture for the \textit{S-GalerkinNet} and the \textit{S-RitzNet}. Each neural network consists of $M+1$ disconnected branches, where each branch of the neural network consists of an input layer with 35 neurons and the \textnormal{swish} activation function, four hidden layers, each of which has 35 \textnormal{swish} activated neurons, and a linear output neuron. 
The \textnormal{ADAM} optimizer uses a decaying learning rate scheme, and the input spatial points are drawn as a continuous stream of quasi-random Sobol points. 
\begin{figure}[t!]
\begin{minipage}{.5\textwidth}
  \begin{center}
  \includegraphics[width=0.99\textwidth]{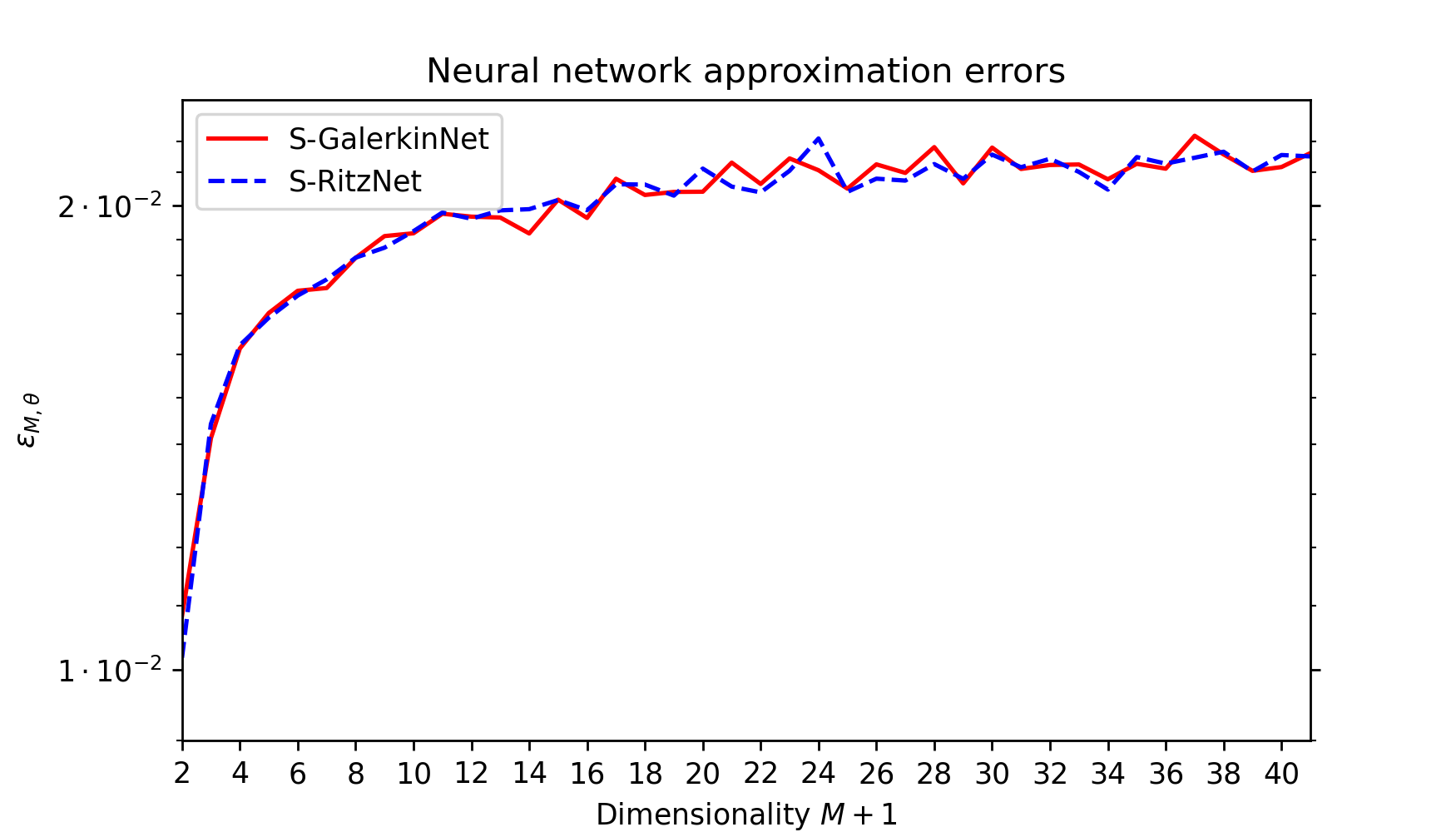}
  \end{center}
\end{minipage}
\begin{minipage}{.5\textwidth}
  \begin{center}
  \includegraphics[width=0.99\textwidth]{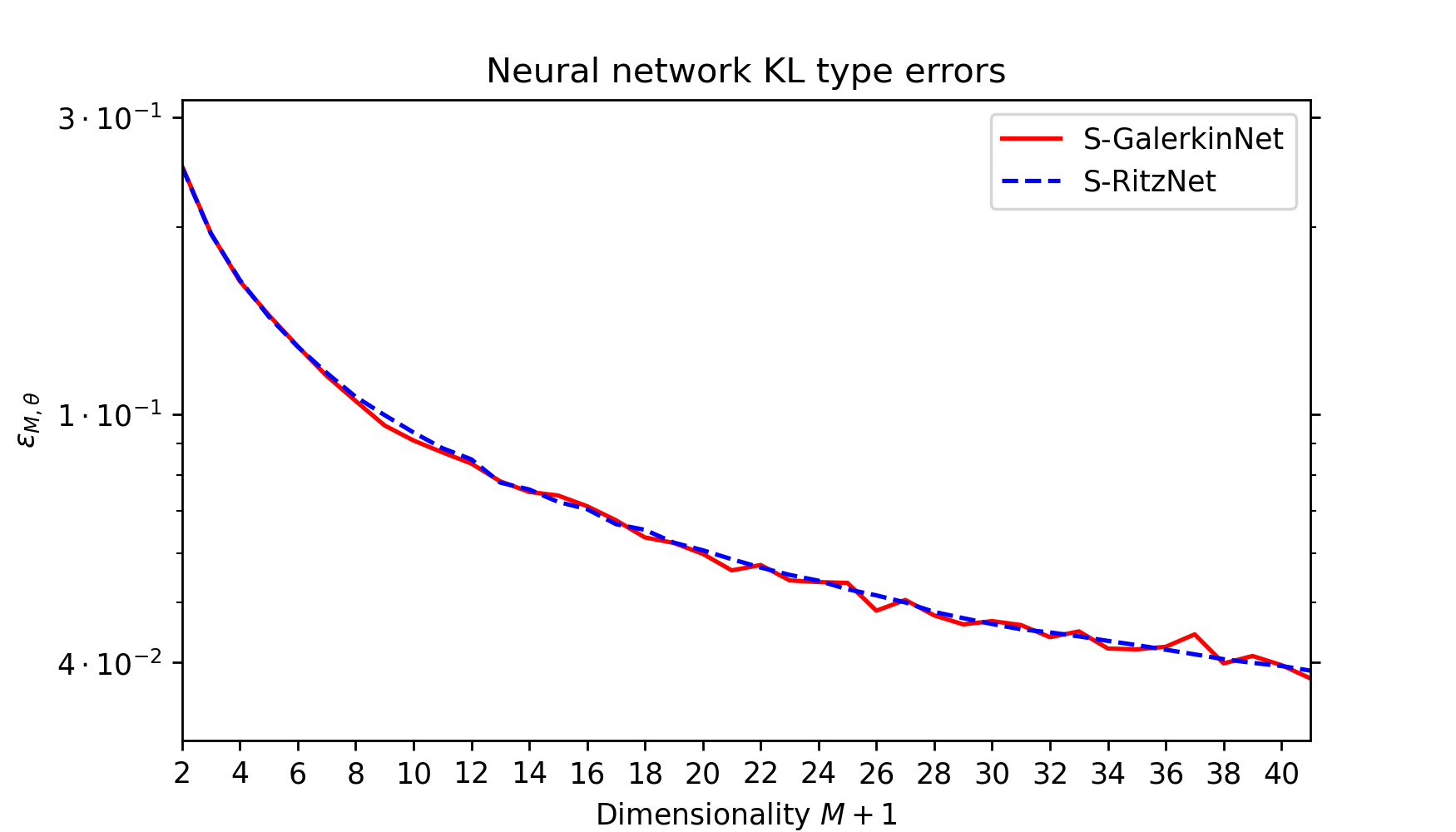}
  \end{center}
\end{minipage}
\begin{center}
\includegraphics[width=0.49\textwidth]{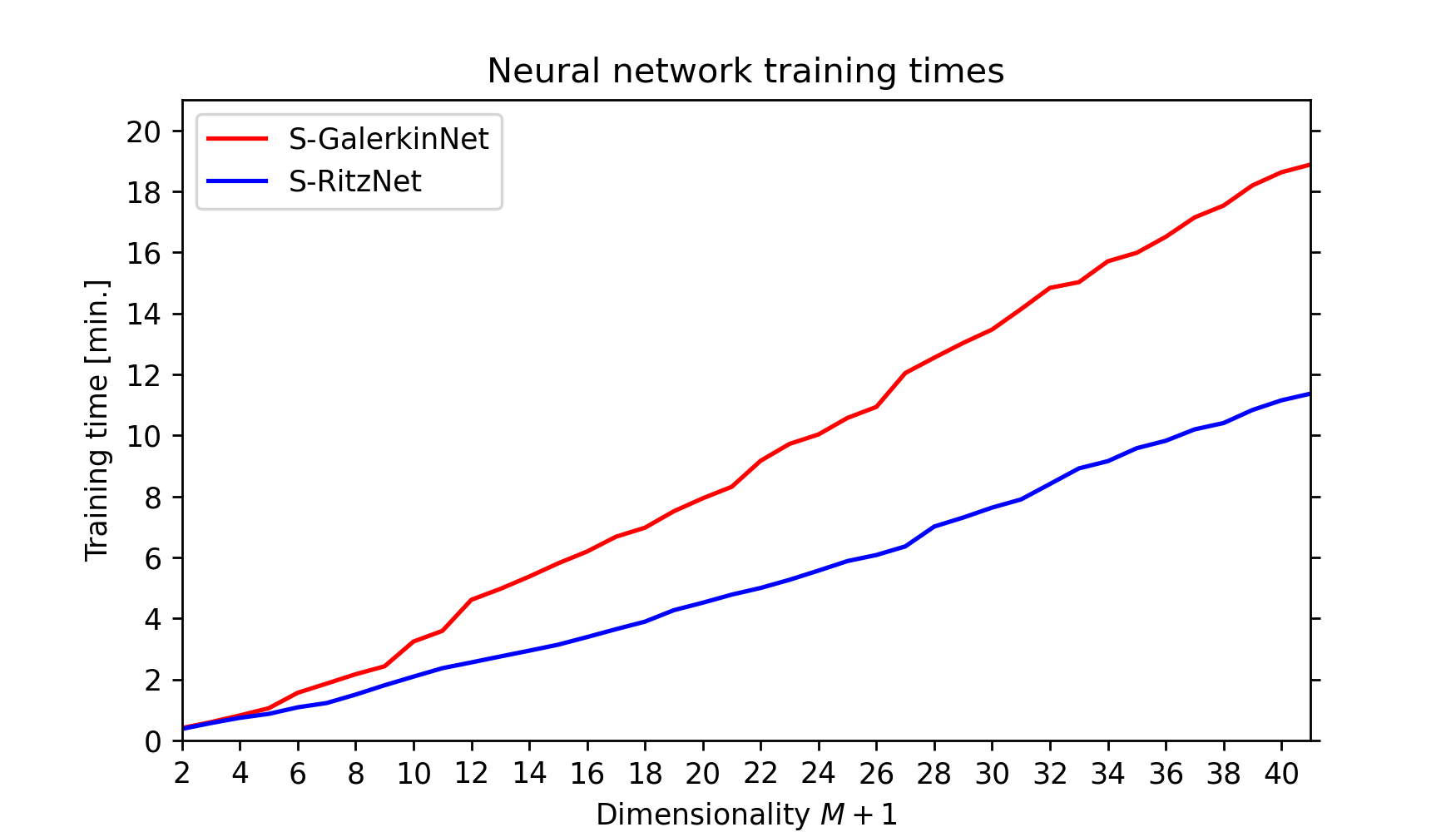}
\end{center}
\caption{Performance and training times of \textit{S-GalerkinNet} and \textit{S-RitzNet}.}
\label{Figure:Experiment2_Data}
\end{figure}
The performance with corresponding training times of the neural networks are given in Figure~\ref{Figure:Experiment2_Data} for $P = 1, \, N = 1, \dots , 40$. The resulting dimensionality of the coupled system of two dimensional PDEs is depicted on the $x$-axis. The errors are calculated against reference solutions, that are generated pathwise via a finite element method. The reference solutions in the shown neural network approximation errors are generated with the same truncation index of the KL-type expansion of the diffusion coefficients as used for the neural networks, to show the performance of the neural networks. The relative $L^2(\Omega; H^1(\DD; \R))$-error ranges from just above $1\%$ for $N = 1$ to just under $2.2\%$ for $N = 40$.
The reference solutions in the shown neural network KL-type errors are generated using $N = 200$ terms in the KL-type expansion of the diffusion coefficients, to show the convergence in the truncation index. By using $40$ KL-terms for the neural network, resulting in a system of $41$ coupled two dimensional PDEs, we achieve a relative $L^2(\Omega; H^1(\DD; \R))$-error of just above $3.8\%$.
The coupling strength $\operatorname{nnz}(A(x)), \operatorname{nnz}(B(x))$, as well as $\operatorname{nnz}(\mathbf{A}(x))$ range from around $13\%$ for $N = 1$ to around $5\%$ for $N = 40$. The relative number of non-negative, non-diagonal entries of the coupling operator decreases with increasing stochastic dimension, since the sparsity of the Galerkin tensor increases exponentially in $N$.
The training time grows linearly with the increasing number of random variables.

As a third example, we consider the elliptic random PDE with a log-normal diffusion coefficient given by a Gaussian process with squared exponential covariance kernel, and constant deterministic forcing term. This diffusion coefficient requires both, a high stochastic dimension as well as a high polynomial dimension. By increasing both parameters, the dimensionality of the resulting fully coupled systems grows extremely fast.
Let $\DD = (0,1)$, and let
\begin{align*}
a(\omega, x) = \exp(\mathcal{G}(\omega, x)), \quad f(\omega, x) = 1,
\end{align*}
denote the stochastic diffusion coefficient and forcing term. Here, $\mathcal{G}$ is a Gaussian random field with zero mean and squared exponential covariance kernel, given by
\begin{align*}
k(x,x') = \exp \Big( - \frac{1}{2} (x - x')^2 \Big).
\end{align*}
We approximate the Gaussian random field by the truncated Karhunen-Loève expansion.
The normalized eigensystem $(\lambda_k, \phi_k)_{k \in \N}$ of the corresponding kernel integral operator is analytically given by 
\begin{align*}
\lambda_k = \Big( \frac{2}{3 + \sqrt{5}} \Big)^{k + 1/2}, \quad \phi_k(x) = 
5^{1/8} \frac{1}{\sqrt{2^k k!}}\exp\Bigg( -\frac{\sqrt{5} - 1}{4} x^2 \Bigg) H_k \Bigg( \bigg( \frac{\sqrt{5}}{2} \bigg)^{1/2} x \Bigg),
\quad k \in \N,
\end{align*}
where $H_k$ denotes the $k$-th physicist's Hermite Polynomial (see e.g. \citep[Ch. 4.3.1]{Rasmussen2006}, \citep[Ch. 4]{Zhu1998}). The Gaussian process is then given by
\begin{align*}
\mathcal{G}(\omega, x) = \sum_{k = 0}^{\infty} \sqrt{\lambda_k} \phi_k(x) Y_k(\omega),
\end{align*}
where $Y_k$ are independent standard normally distributed random variables for $k \in \N$. A truncation of the KL expansion after $N \in \N$ summands implies that our diffusion coefficient is driven by $Y = (Y_0, \dots , Y_{N-1}) : \Omega \to \Gamma = \R^N$, satisfying Assumption \ref{Assumption:RVsReq_gPC_MomentsAndDeterminateMeasure}, and given by
\begin{align*}
a(\omega, x) = \exp \Bigg( \sum_{k = 0}^{N-1} \sqrt{\lambda_k} \phi_k(x) Y_k(\omega) \Bigg).
\end{align*}
Since the log-normal distribution is not determinate, we use tensor product probabilist's Hermite polynomials $(\operatorname{He}_\nu)_{\nu \in \mathcal{I}_{N,P}}$, evaluated at $Y$, as polynomial chaos basis of $L^2(\Gamma, \R)$ with respect to the $N$-variate standard normal distribution, resulting in a fully coupled system of dimension $M + 1 = \frac{(N + P)!}{N!  P!}$. The solution $u: \Omega \times [0,1] \to \R$ of
\begin{align*}
- \frac{d}{dx} \Bigg( \exp\bigg( \sum_{k = 0}^{N - 1} \sqrt{\lambda_k} \phi_k(x) Y_k(\omega) \bigg) \frac{d}{dx} u(\omega, x) \Bigg) 
	&=
1
	&& 
\hspace*{-2.5cm} x \in (0,1), \, \omega \in \Omega,	
	\\
u(\omega, 0) = u(\omega, 1) 
	&=
0
	&&
\hspace*{-2.5cm} \omega \in \Omega,
\end{align*}
is not available in closed form, but the corresponding stochastic Galerkin approximations, with a truncated KL expansion, admit a unique solution, as shown in \citep{Charrier2013}. The stochastic diffusion coefficient, with a truncated KL expansion, satisfies Assumption \ref{Assumption:Requirements_ExUnique_WeakSol_REPDE}, with 
\begin{align*}
a_{\min}(\omega) = \exp\big(-\sqrt{\lambda_0} \phi_0(0) \lVert Y(\omega) \rVert_1\big), \quad a_{\max}(\omega) = \exp\big(\sqrt{\lambda_0} \phi_0(0) \lVert Y(\omega) \rVert_1\big),
\end{align*}
where $a_{\min}^{-1} \in L^p(\Omega; \R)$ for every $p \in (0,\infty)$, since $\exp(\lvert Y_i \rvert) \in L^p(\Omega; \R)$, with respect to the univariate standard normal distribution, for all $p \in (0, \infty)$. 
This also implies that $f \in L^2_{a_{\min}^{-2}}(\Omega; L^2(\DD; \R))$, and the forcing term also satisfies Assumption \ref{Assumption:Requirements_ExUnique_WeakSol_REPDE}. 
Furthermore, we get  $\frac{a_{\max}(\omega) }{a_{\min}(\omega)} = \exp\big(2\sqrt{\lambda_0} \phi_0(0) \lVert Y(\omega) \rVert_1\big) \in L^r(\Omega; \R)$ for all $r \in (1, \infty)$. We choose the enforcer function $e(x):= x(1-x)$, and hence we have
\begin{align*}
{\scriptstyle \mathcal{U}}_{\theta}^{\operatorname{SG}}(x) = x(1-x) \mathcal{N}_\theta^{\operatorname {SG}}(x),
\quad
{\scriptstyle \mathcal{U}}_{\theta}^{\operatorname{SR}}(x) = x(1-x) \mathcal{N}_\theta^{\operatorname {SR}}(x).
\end{align*}
For determining the PC expansion, note that we can fully split the integrals
\begin{align*}
a_\nu(x) 
	= 
\langle a(\cdot,x) , \operatorname{He}_\nu \rangle_{L^2(\Gamma; \R)}
	=
\frac{1}{(2 \pi)^{N/2}} 
	\prod_{i = 0}^{N-1} 
\int_\R \exp\big( \sqrt{\lambda_i} \phi_i(x) y_i \big) \operatorname{He}_{\nu_i} (y_i) \exp\big(-y_i^2/2\big) \, dy_i,
\end{align*}
and hence the spectral coefficients of the deterministic representation $\bar{a}^{(M)}: \Gamma \times \DD \to \R$ are approximated online via a univariate Gauss-Hermite quadrature rule. \\
\begin{figure}[t!]
\begin{minipage}{.5\textwidth}
  \begin{center}
  \includegraphics[width=0.99\textwidth]{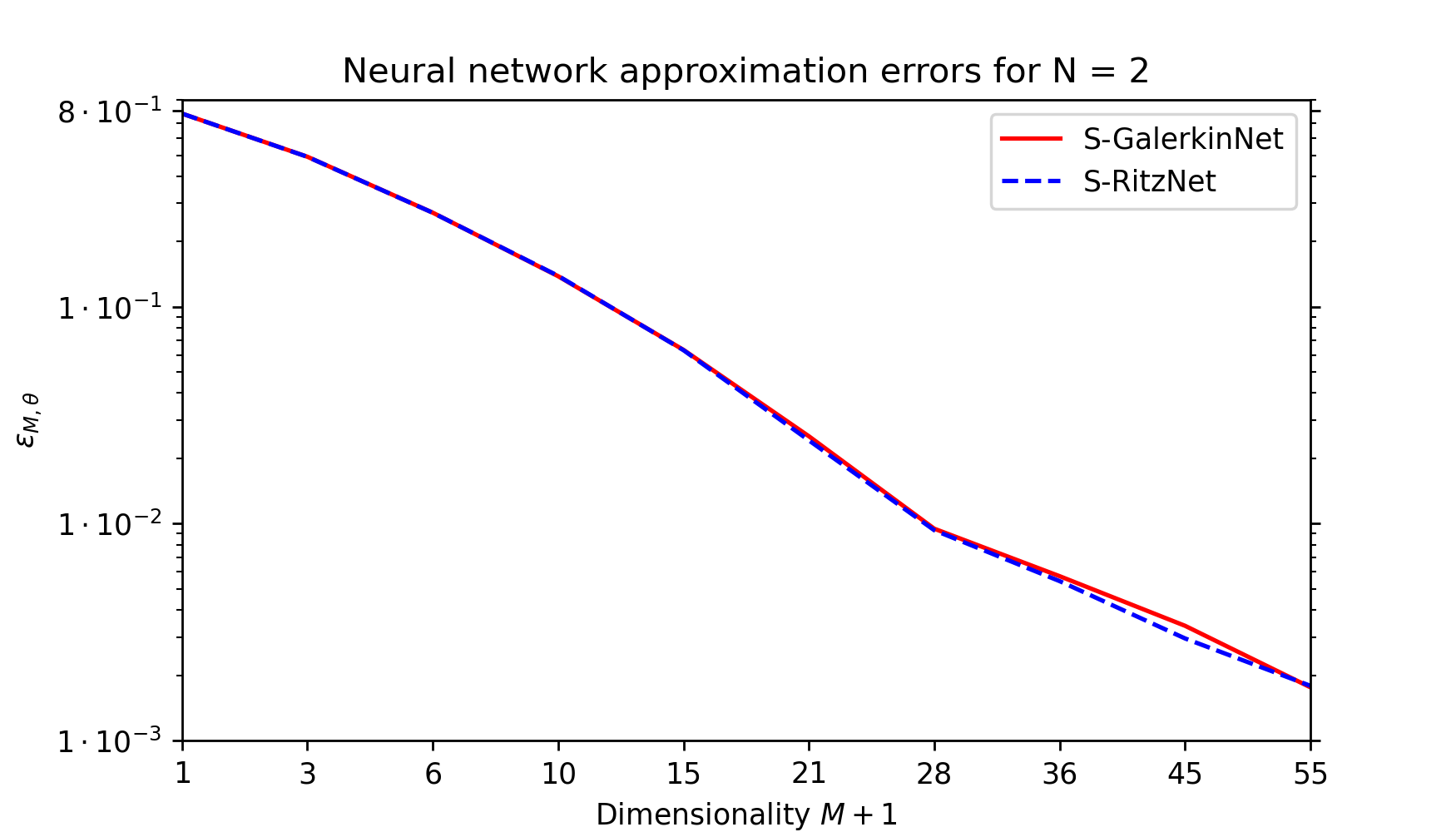} \\
  \includegraphics[width=0.99\textwidth]{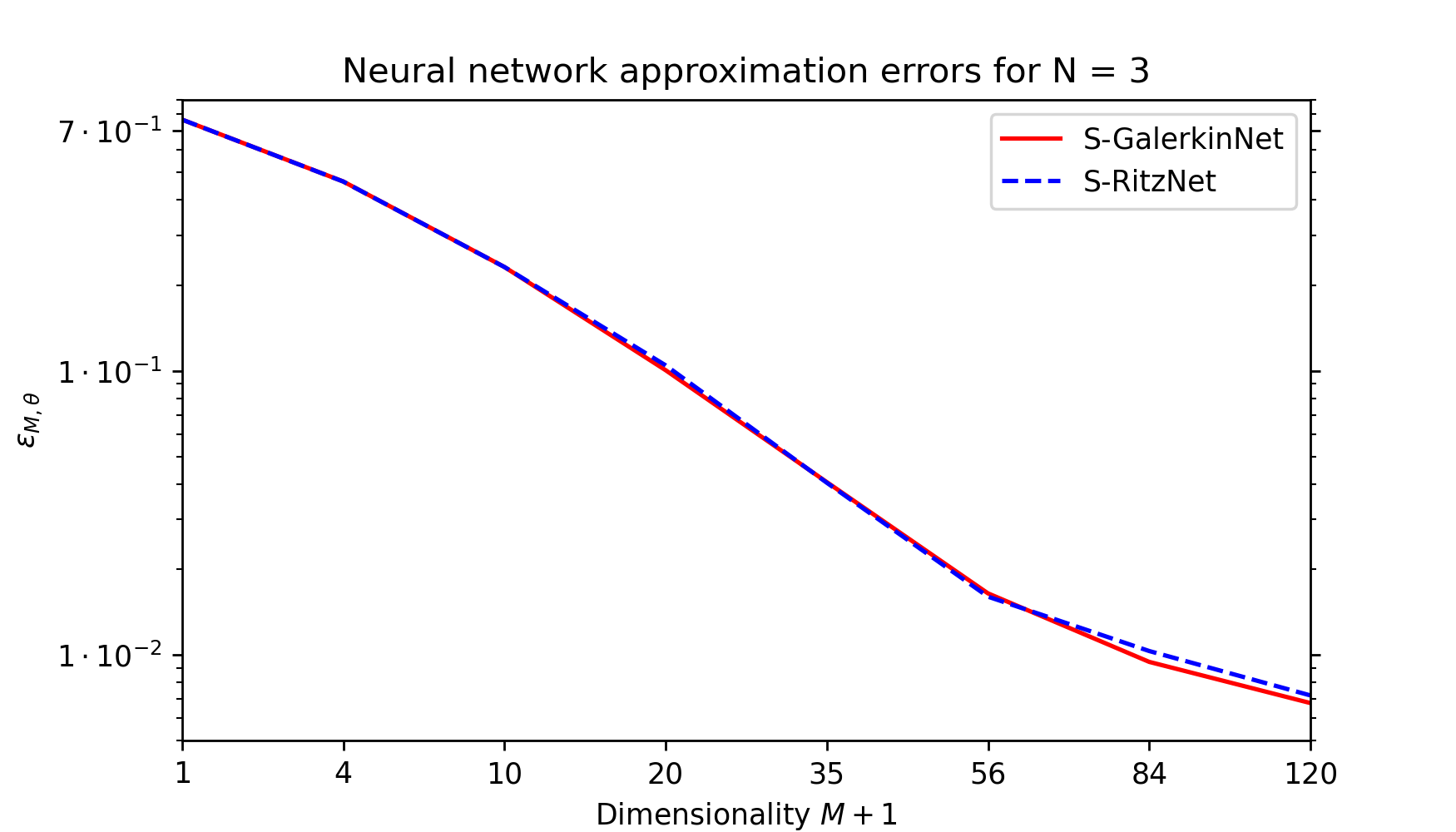}
  \end{center}
\end{minipage}
\begin{minipage}{.5\textwidth}
  \begin{center}
  \includegraphics[width=0.99\textwidth]{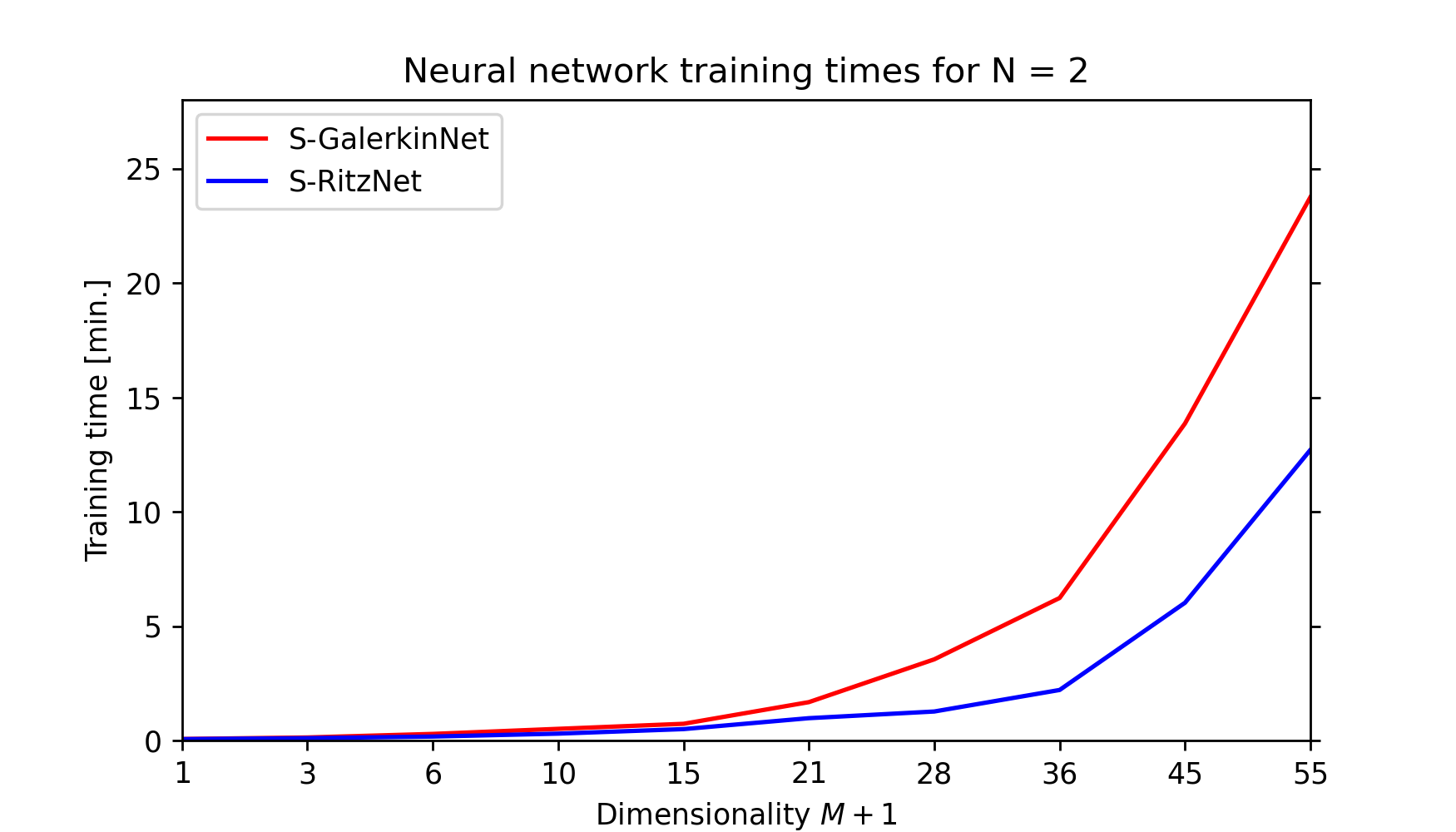} \\
  \includegraphics[width=0.99\textwidth]{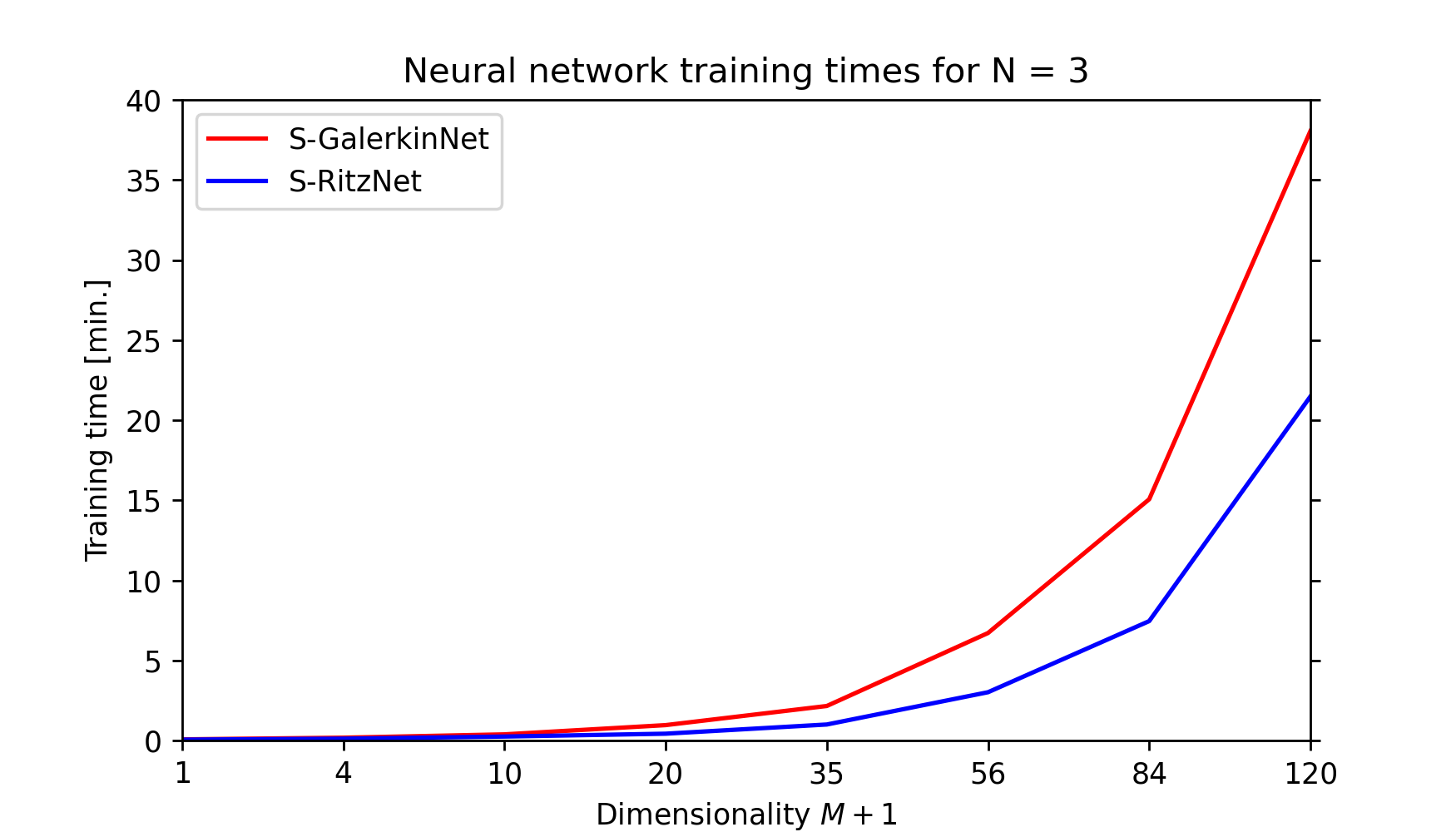}
  \end{center}
\end{minipage}
\caption{Performance and training times of \textit{S-GalerkinNet} and \textit{S-RitzNet}.}
\label{Figure:Experiment3_Data}
\end{figure}
\noindent
The coefficients are sorted according to the graded lexicographic ordering, Definition \ref{Def:graded_lexic_ordering}, to obtain $(a_0, \dots , a_M)$.
The resulting systems are fully coupled, and the corresponding loss functions are given in Problem \ref{Problem:Training_SGN}, \ref{Problem:TrainingSRN}.
We use the same architecture for the \textit{S-GalerkinNet} and the \textit{S-RitzNet}. Each neural network consists of $M+1$ disconnected branches, where each branch of the neural network consists of an input layer with 45 neurons and the \textnormal{swish} activation function, four hidden layers, each of which has 45 \textnormal{swish} activated neurons, and a linear output neuron. 
The \textnormal{ADAM} optimizer uses a decaying learning rate scheme, which is adjusted for different polynomial dimensions, and the input spatial points are drawn as a continuous stream of quasi-random Sobol points. 
The neural networks are scaled to solve an equivalent equation. 
The approximation errors, as well as the training times of the neural networks, are given in Figure \ref{Figure:Experiment3_Data} for $ N = 2, \, P = 0, \dots, 9$, and $N = 3, \, P = 0,  \dots, 7$.
The resulting dimensionality of the coupled system of PDEs is depicted on the $x$-axis for rising polynomial degrees. The error is calculated against reference solutions, that are generated pathwise via a finite element method. The reference solutions use the same truncation index $N = 2, 3$ in their respective KL expansions as the neural network approximations. 
For $N = 2$, we achieve a relative $L^2(\Omega; H^1(\DD; \R))$-error of less than $0.18\%$ for the highest polynomial degree $P = 9$ in a coupled system of 55 equations, and for $N = 3$, we achieve a relative $L^2(\Omega; H^1(\DD; \R))$-error of less than $0.75\%$ for the highest polynomial degree $P = 7$ in a coupled system of 120 equations.
The resulting systems are strongly coupled, where the coupling strength $\operatorname{nnz}(A(x)), \operatorname{nnz}(B(x))$, as well as $\operatorname{nnz}(\mathbf{A}(x))$ range from around $75\%$ for $N = 3, P = 1$ to around $72\%$ for $N = 3, P = 7$.
The training time grows nearly linearly with increasing dimensionality.


\nocite{*}
\bibliography{references_dlsgm}

\end{document}